\documentclass[11pt]{article}
\usepackage[left=1in,right=1in,top=1in,bottom=1in]{geometry}
\usepackage[pdftex]{graphicx}
\usepackage{amsfonts,amssymb,amsmath,amsthm}
\usepackage{epstopdf}
\usepackage{longtable}
\usepackage[normalem]{ulem}
\usepackage{graphicx,multirow,float}
\usepackage{verbatim,float,stackrel}
\usepackage{bbm}
\usepackage[usenames, dvipsnames]{color}
\graphicspath{ {./Figures/} }
\usepackage[export]{adjustbox}
\usepackage[section]{placeins}
\usepackage{enumitem}
\usepackage{bbold}
\usepackage[affil-it]{authblk}
\usepackage[colorlinks,linkcolor=red,citecolor=blue,urlcolor=blue,breaklinks]{hyperref}
\usepackage[ruled, vlined, linesnumbered]{algorithm2e}

\newtheorem{theorem}{Theorem}

\newtheorem{lemma}[theorem]{Lemma}

\def\R#1{(\ref{#1})}
 
\newcommand{\abs}[1]{\left|#1\right|}

\def\D{\,\mathrm{d}}

\DeclareMathOperator*{\argmin}{argmin}
\DeclareMathOperator*{\argmax}{argmax}

\DefineNamedColor{named}{Purple}{cmyk}{0.45,0.86,0,0}

\title{
\vspace{-2em}
High-resolution signal recovery via generalized sampling and functional principal component analysis 
} 

\author{Milana Gataric\footnote{Statistical Laboratory, 
Department for Pure Mathematics and Mathematical Statistics,
 University of Cambridge, Cambridge, UK; e-mail: m.gataric{@}statslab.cam.ac.uk; ORCHID: \href{https://orcid.org/0000-0003-3915-2266}{0000-0003-3915-2266}}
}

\begin{document}

\maketitle

{\abstract{In this paper, we introduce a computational framework for recovering a high-resolution approximation of an unknown function from its low-resolution indirect measurements as well as high-resolution training observations by merging the frameworks of generalized sampling and functional principal component analysis. In particular, we increase the signal resolution via a data driven approach, which models the function of interest as a realization of a random field and leverages a training set of observations generated via the same underlying random process. We study the performance of the resulting estimation procedure and show that high-resolution recovery is indeed possible provided appropriate low-rank and angle conditions hold and provided the training set is sufficiently large relative to the desired resolution. Moreover, we show that the size of the training set can be reduced by leveraging sparse representations of the functional principal components. Furthermore, the effectiveness of the proposed reconstruction procedure is illustrated by various numerical examples.
}}

\vspace{0.5em}


\section{Introduction}

Let $\mathcal{L}_2(D;\mathbb{C}):=\{f:D\mapsto \mathbb{C} : \int_D \abs{f(u)}^2 \D u < \infty \}$ be the space of square-integrable complex-valued functions supported on a compact domain $D\subseteq\mathbb{R}^d$,  with the standard inner product and norm denoted by $\langle\cdot,\cdot\rangle$ and $\|\cdot\|$, respectively.  
Let $f\in\mathcal{L}_2(D;\mathbb{C})$ be a realization of a $\mathcal{L}_2(D;\mathbb{C})$-valued random field $F$ with a probability measure $P$. 
In this paper, we consider the problem of recovering a high-resolution approximation of signal $f$ with respect to the first $p\in\mathbb{N}$ elements of an orthonormal basis $\{\varphi_{\ell}\}_{\ell\in\mathbb{N}}$ in $\mathcal{L}_2(D;\mathbb{C})$ (e.g.~a wavelet basis), from two combined sets of measurements:
\begin{itemize}
\item[(i)] noisy low-resolution measurements of $f$ with respect to the first $q\in\mathbb{N}$ elements of another potentially different Riesz basis $\{\psi_k\}_{k\in\mathbb{N}}$ in $\mathcal{L}_2(D;\mathbb{C})$ (e.g.~a Fourier basis), namely
\begin{equation}\label{eq:observations_Fourier}
\langle f,\psi_k \rangle + w_k, \quad k=1,\ldots,q,
\end{equation}
where  the highest sampled frequency (i.e.~sampling bandwidth) $q$ is relatively small compared to the desired resolution $p$ and $w_k\in\mathbb{C}$ is a realization of a Gaussian noise, as well as
\item[(ii)] noisy high-resolution measurements of a realization $f_1,\ldots, f_n$ of a random sample $F_1,\ldots,F_n$ from the probability measure $P$,  namely training observations that consist of
\begin{equation}\label{eq:observations_coef}
\langle f_i, \varphi_{\ell} \rangle + z_{i\ell}, \quad  \ell=1,\ldots,p, \quad \ i=1,\ldots,n,
\end{equation}
where $z_{i\ell}\in\mathbb{C}$ is a realization of a Gaussian noise.
\end{itemize}

Specifically, we want to recover $f$ in a high-resolution subspace $\mathcal{G}_p:=\text{span}\{\varphi_1,\ldots,\varphi_p\}  \subseteq \mathcal{L}_2(D;\mathbb{C})$, so that its reconstruction achieves the high-resolution approximation rate of $\|f-Q_{\mathcal{G}_p}f\|$, where $Q_{\mathcal{G}_p}f:=\sum_{\ell=1}^p \langle f, \varphi_{\ell} \rangle \varphi_{\ell}$ is the orthogonal projection of $f$ onto $\mathcal{G}_p$ and thus the best possible approximation of $f$ in $\mathcal{G}_p$. It is important to note that normally, such high-resolution rate of approximation cannot be achieved solely from the low-resolution measurements \R{eq:observations_Fourier} of $f$ and typically requires increasing the highest sampled frequency $q=q(p)$ relative to the desired resolution $p$. In this paper, we keep $q$ independent of $p$ and instead increase the size of the training set $n=n(p)$ relative to $p$, thereby leveraging the implicit statistical information given through the high-resolution training observations \R{eq:observations_coef}.
 
To recover $f$  from its low-resolution indirect samples \R{eq:observations_Fourier}, so that the corresponding reconstruction may achieve the high-resolution rate associated with $\mathcal{G}_p$, in this paper, we propose to compute the coefficients of $f$ with respect to the functional principal components constructed from the high-resolution training data \R{eq:observations_coef}. Specifically,  we recover $f$ in a reconstruction subspace $\hat{\mathcal{E}}_m^p\subseteq \mathcal{G}_p$, which is constructed from the first $m$ $p$-dimensional (sparse) eigenvectors of the sample covariance matrix associated with observations \R{eq:observations_coef}, and which estimates the subspace $\mathcal{E}_m:=\text{span}\{\phi_1,\ldots,\phi_m\}\subseteq \mathcal{L}_2(D;\mathbb{C})$ spanned by the first $m$ eigenfunctions ordered by the magnitude of the corresponding eigenvalues, $\lambda_1\geq\lambda_2\geq\cdots$, of the covariance operator associated with the probability measure $P$.

Furthermore, we investigate the conditions 
under which a stable high-resolution reconstruction can be guaranteed for any realization of $F$ and $F_1,\ldots,F_n$. In particular, we show that, in the case of a Gaussian measure $P$ and Gaussian noise, if $m$ and $q=q(m)$ are such that the distance between the subspaces $\mathcal{E}_m$ and $\mathcal{F}_q:=\text{span}\{\psi_1,\ldots,\psi_q\}\subseteq \mathcal{L}_2(D;\mathbb{C})$ is not too large, then the corresponding estimator of $F$ 
is consistent as $m,q/m,p/m,n/(pm)\rightarrow\infty$. Moreover, if $q=q(m)$  and $n=n(p)$ are sufficiently large, the rate of estimation corresponds to the maximum of the two terms, $\mathbb{E}\|F-Q_{\mathcal{E}_m}F\|$ and $\max_{j=1,\ldots,m}\|\phi_j-Q_{\mathcal{G}_p}\phi_j\|$, implying that, if $P$ is a low rank measure so that $\sum_{j=1}^m\lambda_j$ is sufficiently small, then we can achieve the same rate of estimation as the best possible approximation rate in $\mathcal{G}_p$.  Thus, our reconstruction from the low-resolution measurements in $\mathcal{F}_q$ can achieve the high-resolution associated with $\mathcal{G}_p$ as $p$ increases,  only at the price of increasing the size of the training set $n$. 

\subsection{Motivation and relation to previous work}


Reconstructing a function $f$ from the linear functionals~\eqref{eq:observations_Fourier} is an important problem in mathematical signal processing dating back to Shannon~\cite{Shannon1948}, which regained an increased interest over the past decades leading to a boom of areas such as compressed sensing~\cite{Candes2006, Donoho2006} and super-resolution~\cite{Blu2008, Candes2014}. In signal and image processing applications, $f$ represents an unknown audio signal or an image that needs to be recovered from a small amount of its fixed indirect measurements given by a sensing device. For instance, if the measurements are taken with respect to Fourier exponentials, then such problem arises in medical imaging, such as magnetic resonance imaging (MRI), as well as in radar and geophysical imaging; whereas, if the sampling system is a pixel basis (the basis induced by the scaling function of Haar wavelets), then such scenario arises in lens-less optical imaging for example. 



To address such problem, building upon the previous works of \cite{Unser1994, Eldar2003, Grochening2010}, \cite{AH2012,  AHP2013} introduced a computational framework known as Generalized Sampling (GS) that recovers an approximation of an element $f$ of a separable Hilbert space $\mathcal{H}$ with respect to any desired reconstruction basis (or more generally, a frame) in $\mathcal{H}$, from its finitely many functional measurements taken with respect to any other basis in $\mathcal{H}$, such as those given in~\eqref{eq:observations_Fourier}. GS guarantees a noise-robust  reconstruction, which attains the best possible approximation rate in the reconstruction space $\mathcal{G}_p$, provided the distance between the sampling space $\mathcal{F}_q$ and the reconstruction space $\mathcal{G}_p$ is not too large. Such reconstruction was then analyzed for different choices of sampling and reconstruction  spaces, see e.g.~\cite{AHP2014,AGH2014,Ma2015,AGR2019}. The GS framework has also been combined with $\ell_1$-regularization yielding insights into so-called infinite-dimensional compressed sensing \cite{AH2016,AHPR2017}. The results therein established that, if $f$ is sparse with respect to $\mathcal{G}_p$, then by means of $\ell_1$-regularization one still may stably reconstruct $f$ in $\mathcal{G}_p$ even if only randomly sub-sampling in $\mathcal{F}_q$. However, even though by random sub-sampling in $\mathcal{F}_q$ the total number of samples can be substantially reduced, the condition on not too large distance between the spaces $\mathcal{F}_q$ and $\mathcal{G}_p$ remains, meaning that the highest sampled  frequency $q$ has to be large relative to the desired resolution $p$.  In applications such as MRI for example, this may present a time-consuming constraint since, (especially) when under-sampling, high frequencies in the Fourier domain need to be acquired. Also, in applications where fast calibration of an imaging device with respect to non-orthogonal bases is crucial for a real-time operation, such as optical endoscopy for example, time-consuming calibration is typically needed for high-resolution image recovery \cite{Gataric2019}. 

Unlike these previous works, in the present paper, instead of reconstructing a generic deterministic $\mathcal{L}_2$-function,  we model the signal of interest as an observation from a $\mathcal{L}_2$-valued random field with a probability measure $P$ whose structure can be learned  through a training set. Therefore, we can adapt our sampling scheme more closely to the object being sampled, namely to the specific probability measure at hand, and thereby possibly reduce the highest frequency $q$ required for the  high-resolution recovery in $\mathcal{G}_p$. In particular, the reconstruction procedure proposed in this paper, which we call GS-FPCA, combines the aforementioned GS framework with the data-driven approach of Functional Principal Component Analysis (FPCA) from functional data analysis, see e.g.~\cite{RS2005,Hall2006}. By means of FPCA, we construct a suitable reconstruction subspace in $\mathcal{G}_p$ from the training observations \eqref{eq:observations_coef}, thereby circumventing the requirement on the distance between subspaces $\mathcal{F}_q$ and $\mathcal{G}_p$, which is replaced by a condition on the distance between $\mathcal{F}_q$ and the space $\mathcal{E}_m$ spanned by the first $m$ eigenfunctions of the underlying probability measure $P$. Therefore, we can ensure a stable high-resolution reconstruction in $\mathcal{G}_p$ provided the angle between the spaces $\mathcal{F}_q$ and $\mathcal{E}_m$ is positive, even in scenarios when the angle between the spaces $\mathcal{F}_q$ and $\mathcal{G}_p$ is zero.

Previous proposals to use a PCA-regularized reconstruction for increasing image resolution most notably appear
within the problem of face hallucination, the term first coined in the seminal work of \cite{BK2000} in the field of computer vision.  In particular, \cite{Zisserman2001} suggest super-resolving a face image by transferring it from a pixel to an eigenface-domain constructed via a training set of high-resolution images, which was then combined with a face recognition task in \cite{Gunturk2003}. Such technique is also used as the initial step in two-stage super-resolution algorithms that combine a global PCA model with a local patch model, see for example \cite{LSF2007,Ma2010}.
These earlier works operate within a finite-dimensional setting, which could be deduced from the infinite-dimensional model of this paper by constraining $\mathcal{G}_p$ to the $p$-dimensional pixel basis and defining the sampling space as the $q$-dimensional pixel basis, i.e.~$\mathcal{F}_q:=\mathcal{G}_q$, $q<p$. 
In contrast, in this paper  we consider a more general infinite-dimensional framework for computing a stable high-resolution approximation of an unknown object of potentially infinite resolution, which is sampled via a flexible measurement model with respect to any (non-orthogonal) basis, making it applicable to a wider range of practical scenarios. 
Our framework also allows for recovery of sparse representations of the unknown object with respect to different bases, such as wavelets for example, thereby potentially decreasing the required size of the training set. Furthermore, as a result of deploying an infinite-dimensional framework, we provide insights into the conditions on the problem parameters under which it is possible to guarantee that such a procedure succeeds in  high-resolution recovery.

Another notable example of leveraging low-rank structure of the underlying signal being recovered 
appears in acceleration schemes for dynamic MRI 
\cite{Lingala2011,Zhao2012}, 
and more recently for functional MRI \cite{Chiew2016} and  MR fingerprinting \cite{Zhao2018}. There,
typically, a sequence of images over time is reconstructed with respect to the principal components (PCs) estimated from training images with low spatial resolution and high temporal sampling rate. A crucial difference from the approach presented here is that,
instead of increasing the temporal resolution, we are interested in increasing the spatial resolution, and therefore,
we consider PCs estimated from training observations with a high spatial resolution
so that  subsequently we can allow for a high-resolution image recovery from its low-resolution (Fourier) measurements.

The remainder of the paper is organized as follows. Since in this work we leverage GS and FPCA, we dedicate Sections~\ref{sec:gs} and \ref{sec:fpca} to review the main concepts from these frameworks, where in Sections \ref{sub:gs-random} and \ref{sub:fpca-empirical}, we derive additional results used later on. In Section~\ref{sec:gs-fpca}, the proposed GS-FPCA reconstruction method is formulated and its theoretical performance is analyzed with respect to different problem parameters. Additionally, in Sections~\ref{sub:spca} and \ref{sub:rr} we describe variants of GS-FPCA that arise due to the regularization techniques of sparse PCA and ridge regression. In Section~\ref{sec:sim}, the empirical performance of GS-FPCA is investigated in different simulation scenarios. Specifically, in Section~\ref{sub:1D} we use a 1D generative model, while in Section~\ref{sub:2D}, we use 2D brain-phantom images. In Section~\ref{sec:conclusions}, we conclude with discussions and future work.  

A summary of notation used throughout the paper is provided in Table~\ref{Tab:Notation}.

{
\begin{table}[H]
\begin{center}
\begin{tabular}{c|ccl}
Symbol & $\subseteq$ & Basis  & Description \\ \hline
$\mathcal{F}_q$ & $\mathcal{L}_2(D;\mathbb{C})$ & $\{\psi_k\}_{k=1}^q$ &  Low-resolution space where unknown $f$ is sampled \\
$\mathcal{G}_p$ & $\mathcal{L}_2(D;\mathbb{C})$ & $\{\varphi_{\ell}\}_{\ell=1}^p$ &  High-resolution space where training set $\{f_i\}_{i=1}^{n}$ is sampled \\
$\mathcal{E}_m$ & $\mathcal{L}_2(D;\mathbb{C})$ & $\{\phi_j\}_{j=1}^m$ &  Principal eigenspace associated with probability measure $P$ \\
$\hat{\mathcal{E}}_m^p$ & $\mathcal{G}_p$ & $\{\hat\phi_j\}_{j=1}^m$ & 
Reconstruction space for $f$ computed from $\{f_i\}_{i=1}^{n}$ \\ \hline
\end{tabular}
\caption{Notation of different subspaces.}
\label{Tab:Notation}
\end{center}
\end{table}
}




\section{Generalized Sampling (GS)}\label{sec:gs}

Given measurements $\{\langle f,\psi_k \rangle\}_{k=1}^q$ of an unknown function $f\in\mathcal{L}_2(D;\mathbb{C})$ with respect to the first $q$ elements of a basis  $\{\psi_k\}_{k\in\mathbb{N}}$ in $\mathcal{L}_2(D;\mathbb{C})$, GS recovers $f$ with respect to the first $p$ elements of any desired, potentially different basis $\{\varphi_{\ell}\}_{\ell\in\mathbb{N}}$ in $\mathcal{L}_2(D;\mathbb{C})$. Specifically, if $\mathcal{G}_p:=\text{span}\{\varphi_1,\ldots,\varphi_p\}\subseteq\mathcal{L}_2(D;\mathbb{C})$ denotes the desired reconstruction space and $\mathcal{F}_q:=\text{span}\{\psi_1,\ldots,\psi_q\}\subseteq\mathcal{L}_2(D;\mathbb{C})$ denotes the given sampling space, and if  $Q_{\mathcal{G}_p}f$ and  $Q_{\mathcal{F}_q}f$ denote the orthogonal projections of $f$ to the respective subspaces, then the GS reconstruction 
\begin{equation}\label{eq:GSorig}
\tilde f_{\mathrm{GS}} := \sum_{\ell=1}^p \tilde{a}_{\ell}  \varphi_{\ell}\in\mathcal{G}_p, 
\end{equation} 
is defined so that it satisfies condition 
$\langle Q_{\mathcal{F}_q} \tilde f_{\mathrm{GS}}, \varphi_{\ell}  \rangle = \langle Q_{\mathcal{F}_q} f,  \varphi_{\ell} \rangle$, $\ell=1,\ldots,p$. Equivalently, the coefficients $\{\tilde{a}_{\ell}\}_{\ell=1}^{p}$ of the GS reconstruction $\tilde f_{\mathrm{GS}}$ correspond to the least-square solution of the linear system
\begin{equation}\label{eq:GSsystem}
\begin{pmatrix} \langle \varphi_1, \psi_1 \rangle & \cdots & \langle \varphi_p, \psi_1 \rangle  \\ \vdots & & \vdots \\  \langle \varphi_1, \psi_q \rangle & \cdots & \langle \varphi_p, \psi_q \rangle  \end{pmatrix} \begin{pmatrix} a_1\\ \vdots\\ a_p \end{pmatrix} = \begin{pmatrix} \langle f,\psi_1 \rangle \\ \vdots \\ \langle f,\psi_q \rangle \end{pmatrix},
\end{equation}
i.e. they can be computed as 
$
\argmin_{\{a_j\}_{j=1}^p\subseteq\mathbb{C}} \sum_{k=1}^q \bigl| \langle f,\psi_k \rangle - \sum_{\ell=1}^p a_{\ell} \langle\varphi_{\ell},\psi_k\rangle \bigr|^2.
$
By the results of \cite{AHP2013} 
we know that,
if 
$$
\cos\angle(\mathcal{G}_p,\mathcal{F}_q):=\inf_{\{g\in\mathcal{G}_p: \|g\|=1\}} \|Q_{\mathcal{F}_q}g\| 
>0,
$$
then
for any $f\in\mathcal{L}_2(D;\mathbb{C})$ there exists a unique reconstruction $\tilde f_{\mathrm{GS}}$, which satisfies the sharp bound
\begin{equation}\label{eq:GSbound}
\|\tilde f_{\mathrm{GS}} - f \| \leq \sec\angle(\mathcal{G}_p,\mathcal{F}_q) \|Q_{\mathcal{G}_p}f - f\|.
\end{equation}
Moreover,  for any fixed $p$ and arbitrarily small $\epsilon>0$, the angle condition $\cos\angle(\mathcal{G}_p,\mathcal{F}_q)\geq \epsilon$ is satisfied for any sufficiently large $q=q(p,\epsilon)$, and thus 
 $\tilde{f}_{\mathrm{GS}}$ achieves the best possible approximation rate in $\mathcal{G}_p$ up to a constant.  
Also, the condition number of such reconstruction, which is defined to indicate reconstruction stability to measurement perturbations $\langle f+g,\psi_k \rangle$, $g\in \mathcal{L}_2(D;\mathbb{C})$, is  
proportional to $\sec\angle(\mathcal{G}_p,\mathcal{F}_q)$. 
The work of \cite{AHP2013} further shows that, if $\{\psi_k\}_{k\in\mathbb{N}}$ is a Riesz basis with Riesz constants $r_1,r_2>0$ such that
\begin{equation}\label{eq:rieszbounds}
r_1 \|b\|_{\ell_2} \leq \Bigl\|\sum_{k\in\mathbb{N}} b_k \psi_k \Bigr\|  \leq r_2 \|b\|_{\ell_2}, \quad \forall b=\{b_k\}_{k\in\mathbb{N}} \in \ell_2(\mathbb{N}),
\end{equation}
and $\{\varphi_{\ell}\}_{\ell\in\mathbb{N}}$ is an orthonormal basis, then $\sec\angle(\mathcal{G}_p,\mathcal{F}_q) \leq \sqrt{r_2}/\sigma_{\min}(A_{p,q}) \leq \sqrt{r_2/r_1} \sec\angle(\mathcal{G}_p,\mathcal{F}_q)$, where $\sigma_{\min}(A_{p,q})$  denotes  the minimal singular value of the system matrix $A_{p,q}$ in \eqref{eq:GSsystem}, namely $\sigma_{\min}(A_{p,q}):=\lambda_{\min}(A_{p,q}^*A_{p,q})^{1/2}$, where $\lambda_{\min}$ is the minimal eigenvalue and $A_{p,q}^*$ is the adjoint of $A_{p,q}$. Note that $r_1=r_2=1$ when $\{\psi_k\}_{k\in\mathbb{N}}$ is an orthonormal basis.

We remark that, alternatively, the angle condition can be interpreted so that for any fixed $q$ and $\epsilon$, resolution $p=p(q,\epsilon)$ needs to be sufficiently small. As we decrease the number of measurements $q$  we also need to decrease resolution $p$ so that the angle condition is satisfied, but the rate at which this happens depends on the specific choices of spaces $\mathcal{G}_p$ and $\mathcal{F}_q$, and has been analyzed in a variety of settings, see e.g.~\cite{AHP2014,AGH2014,AGR2019}. In particular, if  $\mathcal{F}_q$ is spanned by a Fourier basis or frame, it is known that this rate is linear when $\mathcal{G}_p$ is spanned by wavelets, and quadratic when $\mathcal{G}_p$ is spanned by polynomials. 

\subsection{Generalized sampling with random noise}\label{sub:gs-random}

In what follows, we consider the error bound \eqref{eq:GSbound} when  the measurements of $f$ are perturbed by random noise. To this end, let us assume that the measurements are $\{\langle f,\psi_k \rangle + W_k \}_{k=1}^q$, where $W_1,\ldots,W_q$ are i.i.d.~Gaussian random variables in $\mathbb{C}$ with mean zero and variance $\sigma^2$, i.e.~$\mathrm{Re}(W_1), \mathrm{Im}(W_1),\ldots,\mathrm{Re}(W_q), \mathrm{Im}(W_q)$ are i.i.d.~Gaussian random variables in $\mathbb{R}$ with mean zero and variance $\sigma^2/2$.
Let us now define 
\begin{equation}\label{eq:GS}
\hat{f}_{\mathrm{GS}}:=\sum_{\ell=1}^p \hat{a}_{\ell}  \varphi_{\ell},
\end{equation}
where $\{\hat{a}_j\}_{j=1}^{p} := \argmin_{\{a_j\}_{j=1}^p\subseteq\mathbb{C}} \sum_{k=1}^q \bigl| \langle f,\psi_k \rangle + W_k - \sum_{\ell=1}^p a_{\ell} \langle\varphi_{\ell},\psi_k\rangle \bigr|^2$. For simplicity, let $\{\psi_k\}_{k\in\mathbb{N}}$ be a Riesz basis such that \eqref{eq:rieszbounds} holds and $\{\varphi_{\ell}\}_{\ell\in\mathbb{N}}$ an orthonormal basis, so that we can use $\sigma_{\min}(A_{p,q})\geq \sqrt{r_1} \cos\angle(\mathcal{G}_p,\mathcal{F}_q)$ as well as $\|\hat f_{\mathrm{GS}} - \tilde f_{\mathrm{GS}} \| = \|\hat a - \tilde a \|_2$, where $A_{p,q}$ denotes the system matrix in \eqref{eq:GSsystem}, $\tilde f_{\mathrm{GS}}$ is defined in \eqref{eq:GSorig} and $ \|\hat a - \tilde a \|_2:=\bigl(\sum_{j=1}^p |\hat a_j - \tilde a_j|^2\bigr)^{1/2}$. Since  
$$
\sigma_{\min}(A)=\sigma_{\min}
\begin{pmatrix}
\mathrm{Re}(A) & -\mathrm{Im}(A) \\
 \mathrm{Im}(A) & \mathrm{Re}(A)
\end{pmatrix}
$$
holds for any complex-valued matrix $A$,
by the finite-sample bound for the least squares estimator, see e.g.~\cite{Hsu2012}, for any $\delta_0>0$ we have $\mathbb{P}\bigl\{  \sigma_{\min}(A_{p,q}) \|\hat a - \tilde a \|_2 > \sigma \sqrt{(2p+2\log(1/\delta_0) )/ q}   \bigr\} \linebreak < \delta_0$. 
Thus, if $\cos\angle(\mathcal{G}_p,\mathcal{F}_q)>0$,
then with probability at least $1-\delta_0$, $\hat f_{\mathrm{GS}}$ satisfies
\begin{equation}\label{GSprobbound}
\|\hat f_{\mathrm{GS}} - f \| \leq   \sec\angle(\mathcal{G}_p,\mathcal{F}_q) \Bigl\{\| Q_{\mathcal{G}_p^{\perp}}f \| + \sigma \sqrt{(2p+2\log(1/\delta_0) )/ (q r_1)}  \Bigr\}.
\end{equation}
Moreover, similarly to the approach by \cite{Cohen2013}, if we assume a uniform bound on $f$, that is, for a $\tau>0$ we consider functions $f\in \mathcal{L}_2(D;\mathbb{C})$ such that $\sup_{u\in D} |f(u)| \leq\tau$, and define a truncation operator 
\begin{equation}\label{trunc}
T_{\tau}(g) := \mathrm{sign}(g) \min\{ |g|, \tau\}, \quad g\in \mathcal{L}_2(D;\mathbb{C}),
\end{equation}
so that we may use $\|T_{\tau}(g)-f\|\leq \min\{\|g-f\|, 2\tau \Delta \}$, where $\Delta:=\sqrt{\int_{D}\D u}$,  then from the high probability bound in \eqref{GSprobbound} we obtain the expectation bound
$$
\mathbb{E} \|T_{\tau}(\hat f_{\mathrm{GS}}) - f \| \leq \sec\angle(\mathcal{G}_p,\mathcal{F}_q) \Bigl\{\| Q_{\mathcal{G}_p^{\perp}}f \| + \sigma \sqrt{(2p+2\log(1/\delta_0) )/ (q r_1)}  \Bigr\} + 2\tau \Delta \delta_0.
$$
Furthermore, due to \cite{Mallat2008}, we know that if $\mathcal{G}_p$ is the subspace spanned by the boundary-corrected Daubechies wavelets with $s$ vanishing moments and $f$ is $\gamma$-H\"{o}lder continuous, $\gamma\in(0,s)$, then $\| Q_{\mathcal{G}_p^{\perp}}f \| = \mathcal{O} (p^{-\gamma})$. Thus, in this case, for $\delta_0:=e^{-p}$ and $\epsilon>0$, if $p$ and $q$ are such that $\cos\angle(\mathcal{G}_p,\mathcal{F}_q)>\epsilon$ and $e^{-p}\tau\Delta \lesssim p^{-\gamma}+  \sigma \sqrt{p/q}$,  then   
\begin{equation}\label{GSorder}
\mathbb{E} \|T_{\tau}(\hat f_{\mathrm{GS}}) - f \|  = \mathcal{O} \Bigl( 1 / p^{\gamma} + \sigma \sqrt{p/q} \Bigr).
\end{equation}


\section{Functional Principal Component Analysis (FPCA)}\label{sec:fpca}

If $F$ is a random field with probability measure $P$ on $\mathcal{L}_2(D;\mathbb{C})$ with mean $\mu(u):=\mathbb{E}[F(u)]$ and covariance $K(u,v):=\mathbb{E}[(F(u)-\mu(u))\overline{(F(v)-\mu(v))}]$, $u,v\in D$, then
by Mercer's lemma, there exist a non-increasing sequence of non-negative eigenvalues $\lambda_1\geq\lambda_2\geq\cdots\geq0$ and an orthonormal sequence of eigenfunctions  $\{\phi_j\}_{j\in\mathbb{N}}$ of the covariance operator $K$ such that $\int_{D} K(u,v) \phi_j(u) \D u = \lambda_j \phi_j(v)$,  $K(u,v)=\sum_{j\in\mathbb{N}} \lambda_j \phi_j(u) \overline{\phi_j(v)}$ and such that 
\begin{equation}\label{eq:KLexp}
F=\mu + \sum_{j\in\mathbb{N}} \sqrt{\lambda_j} \xi_j \phi_j,
\end{equation}
where $\xi_j:=\lambda_j^{-1/2}\langle F-\mu,\phi_j \rangle$ are uncorrelated random variables with zero mean and unit variance. Moreover, if $F$ is a Gaussian field, then $\xi_j$ are standard Gaussian random variables. Eigenfunctions $\{\phi_j\}_{j\in\mathbb{N}}$ are also known as functional principal components (FPCs) of $F$ and the expression \eqref{eq:KLexp} is known as the Karhunen-Loeve (KL) expansion of $F$, see for example \cite{RS2005}. Such representation of $F$ is known to be optimal in the following sense:
\begin{equation}\label{eq:KLopt}
\{\phi_j\}_{j=1}^{m} = \argmin_{\{\{\varphi_j\}_{j=1}^m:\langle\varphi_j,\varphi_k\rangle=\delta_{jk}\}} \mathbb{E}  \| F-\mu - \sum_{j=1}^m\langle F-\mu,\varphi_j \rangle \varphi_j \|^2, 
\end{equation}
for any $m\in\mathbb{N}$, where $\delta_{jk}=1$ if $j=k$ and zero otherwise. 


\subsection{Empirical high-resolution functional principal components}\label{sub:fpca-empirical}

Since in practice we observe only finitely many noisy coefficients of $F$ with respect to the first $p$ elements of an orthonormal basis $\{\varphi_j\}_{j\in\mathbb{N}}$, 
let us now consider the finite-dimensional high-resolution subspace $\mathcal{G}_p:=\text{span}\{\varphi_1,\ldots,\varphi_p\}\subseteq\mathcal{L}_2(D;\mathbb{C})$ and let $Q_{\mathcal{G}_p}$ denote the orthogonal projection onto $\mathcal{G}_p$. 
First, consider a $\mathcal{G}_p$-valued random variable $Q_{\mathcal{G}_p}F=\sum_{j=1}^{p}\langle  F ,\varphi_j  \rangle \varphi_j$, whose mean is denoted by  $\mu_p:=\mathbb{E}[Q_{\mathcal{G}_p}F]=Q_{\mathcal{G}_p}\mu$ and covariance $K_p(u,v) := \mathbb{E}[(Q_{\mathcal{G}_p}F(u)-\mu_p(u))\overline{(Q_{\mathcal{G}_p}F(v)-\mu_p(v))}]$, 
with the corresponding eigenfunctions and eigenvalues denoted by $\{\phi_j^p\}_{j=1}^{p}$ and $\{\lambda_j^p\}_{j=1}^p$, respectively. 
If we now define a $\mathbb{C}^p$-valued random variable 
$$
X:=X(F)=(\langle  F ,\varphi_1  \rangle, \ldots, \langle F ,\varphi_p  \rangle)^{\top},
$$
we see that its mean vector $\mu_{X}:= \mathbb{E}[X]$ is equal to $(\langle  \mu,\varphi_1  \rangle, \ldots, \langle  \mu,\varphi_p  \rangle)^{\top}$ and its covariance matrix $\Sigma_X:=\mathbb{E} [(X - \mu_{X}) \overline{(X -\mu_{X})}]$ satisfies
$K_p(u,v) = (\varphi_1(u), \ldots, \varphi_p(u))  \Sigma_X (\overline{\varphi_1(v)}, \ldots, \overline{\varphi_p(v)} )^{\top}.$
Writing 
$$
e^p_j:=(\langle \phi_j^p ,\varphi_1 \rangle,\ldots, \langle \phi_j^p ,\varphi_p \rangle)^{\top},\quad j=1,\ldots,p,
$$
it then follows that $\Sigma_X e^p_j = \lambda^p_j e^p_j$ and $\Sigma_X^{(k\ell)} 
= \sum_{j=1}^p \lambda^p_j \langle \phi_j^p ,\varphi_k \rangle \overline{\langle \phi_j^p ,\varphi_{\ell} \rangle}$, $k,\ell=1,\ldots,p$. 
Moreover, if $F$ is a Gaussian random field, then $X$ is a multivariate Gaussian with mean $\mu_X$ and covariance $\Sigma_X$, since any finite-dimensional section of a Gaussian process is a multivariate Gaussian.


We can now model the training observations \eqref{eq:observations_coef} as realizations of i.i.d.~multivariate random variables  
\begin{equation}\label{eq:Ys}
Y_i := Y_i(F_i,Z_i) = X_i(F_i) + Z_i, \quad i=1,\ldots,n,
\end{equation}
where $X_i:=X_i(F_i)=(\langle  F_i ,\varphi_1  \rangle, \ldots, \langle F_i ,\varphi_p  \rangle)^{\top}$, $F_1,\ldots,F_n\sim^{\mathrm{iid}} P$ and $Z_1,\ldots,Z_n$ are i.i.d.~Gaussian on $\mathbb{C}^p$ with mean zero and covariance $\tilde\sigma^2I_p$, which are also independent of $F_1,\ldots,F_n$.  
If $P$ is Gaussian, then  $Y_1,\ldots,Y_n$ are Gaussian 
with mean $\mu_Y=\mu_X$ and covariance $\Sigma_Y=\Sigma_X + \tilde\sigma^2 I_p$, in which case
it is known that the eigenvectors $\hat e^p_1, \ldots, \hat e^p_p$ of the sample covariance matrix $\hat\Sigma_{Y}:=n^{-1}\sum_{i=1}^n (Y_i-\hat{\mu}_{Y})\overline{(Y_i-\hat{\mu}_{Y})}$,  where $\hat{\mu}_{Y}:=n^{-1}\sum_{i=1}^n Y_i$, are consistent estimators of the eigenvectors $e^p_1,\ldots,e^p_p$ of $\Sigma_X$ as $p/n\rightarrow0$, e.g.~\cite{Koltchinskii2017}. 
Moreover, by utilizing the classical results of the Galerkin method, e.g.~\cite{BO1987}, we can obtain the following high-probability bound on the distance between the space spanned by the eigenfunctions at the population level, $\mathcal{E}_m:=\text{span}\{\phi_1,\ldots,\phi_m\}$, and the space spanned by the high-resolution empirical eigenfunctions, $\hat{\mathcal{E}}^p_m:=\text{span}\{\hat\phi_1^p,\ldots,\hat\phi_m^p\}$, where
 $\hat\phi_j^p:=(\varphi_1,\ldots,\varphi_p)\hat e_j^p$.

\begin{lemma}\label{lem:eig_mean_est} Let $P$ be  a Gaussian measure on $\mathcal{L}_2(D;\mathbb{C})$ with mean $\mu$, eigenfunctions $\{\phi_j\}_{j\in\mathbb{N}}$ and eigenvalues $\{\lambda_j\}_{j\in\mathbb{N}}$, and let $\{\varphi_j\}_{j\in\mathbb{N}}$ be an orthonormal basis in $\mathcal{L}_2(D;\mathbb{C})$. For any $m\in\mathbb{N}$ and $p\geq m$, let $\mathcal{E}_m:=\text{span}\{\phi_1,\ldots,\phi_m\}$,  $\mathcal{G}_p:=\text{span}\{\varphi_1,\ldots,\varphi_p\}$,  $\epsilon_{p}:=\max_{j=1,\ldots,m} \| Q_{\mathcal{G}_p^{\perp}}\phi_j \|$ and $\epsilon_{p}':=\|Q_{\mathcal{G}_p^{\perp}}\mu \|$. For any $n\geq p$, let $F_1,\ldots,F_n\sim^{\mathrm{iid}} P$ and let $Y_1,\ldots,Y_n$ be as in \eqref{eq:Ys}, and also define  $\hat\mu_p:=(\varphi_1,\ldots,\varphi_p) \hat\mu_Y$,   $\hat\phi_j^p:=(\varphi_1,\ldots,\varphi_p)\hat e_j^p$, and $\hat{\mathcal{E}}^p_m:=\text{span}\{\hat\phi_1^p,\ldots,\hat\phi_m^p\}$. Then 
\begin{itemize}
 \item[$(a)$]  there exist $C,\tilde{C}$  and $p_0$ such that for any $p\geq \max\{p_0,m\}$ and $\delta\in[2e^{-n},1)$ with probability at least $1-\delta$ we have
$$\sin\angle(\mathcal{E}_m,\hat{\mathcal{E}}^p_m) \leq  C \epsilon_{p} \sqrt{\frac{m}{2}}  + \frac{2 \tilde{C}\sqrt{m}(\lambda_1+\tilde\sigma^2)}{\lambda_m-\lambda_{m+1} - C \lambda_m^2 \epsilon_{p}^2} \biggl(\sqrt{\frac{p}{n}}+\sqrt{\frac{1}{n}\log\frac{2}{\delta}}\biggl)  =: \tilde\epsilon_{mpn\delta},$$
provided that $\lambda_{m}-\lambda_{m+1}>C \lambda_m^2 \epsilon_{p}^2 + 2\tilde{C} (\lambda_1+\tilde\sigma^2)(\sqrt{p/n}+\sqrt{\log(2/\delta)/n})$, 
 \item[$(b)$] for any $\delta'\in(0,1)$ with probability at least $1-\delta'$ we have
$$\|\mu  -\hat\mu_p \| \leq \epsilon_{p}'  +  \sqrt{\lambda_1+\tilde\sigma^2} \biggl( \sqrt{\frac{p}{n}} + \sqrt{\frac{2}{n}\log\frac{1}{\delta'} } \biggr)  =: \bar\epsilon_{pn\delta'}.$$
\end{itemize}
\end{lemma}

The proof of Lemma~\ref{lem:eig_mean_est} is given in Appendix~\ref{app:proofs}. We now discuss the order of bounds $\tilde\epsilon_{mpn\delta}$ and $\bar\epsilon_{pn\delta'}$ derived in this lemma, since these play an important role later on. First of all, observe that the order of the second summand in these bounds is $\sqrt{mp/n}$ and $\sqrt{p/n}$ respectively, provided $\delta,\delta'\geq 2e^{-p/2}$ and provided the eigenvalue gap $\lambda_{m}-\lambda_{m+1}$ is lower-bunded as stated in Lemma~\ref{lem:eig_mean_est}, the latter one being a typical assumption required for consitensy of PCA estimation \cite{Koltchinskii2017,Ma2013}. 
Moreover, 
if additionaly 
$P$ is a probability measure on
the space of $\gamma$-H\"{o}lder continuous functions and $\mathcal{G}_p$ is the $p$-dimensional space of boundary-corrected wavelets with $s>\gamma$ vanishing moments, then $\max\{\epsilon_p,\epsilon_p'\}=\mathcal{O}(p^{-\gamma})$, and therefore, 
we have 
\begin{equation}\label{PCAorder}
\max\{\tilde\epsilon_{mpn\delta},\bar\epsilon_{pn\delta'}\} = \mathcal{O} \Bigl( \sqrt{m} / p^{\gamma} + (\lambda_1+\tilde{\sigma}^2)\sqrt{mp/n}\Bigr).
\end{equation}
Note that such bound improves with increasing $p$, provided $n$ is also increasing. In particular,
if $n\gtrsim p^{2\gamma+1}$, then we can obtain the $\mathcal{G}_p$-rate of approximation, $p^{-\gamma}$, up to factor $\sqrt{m}$, namely for such $n$ we have the bound of order $\sqrt{m} / p^{\gamma}$.

\section{GS-FPCA reconstruction method}\label{sec:gs-fpca}

In this section, we introduce and analyze a method for computing an estimate of an unknown $\mathcal{L}_2(D;\mathbb{C})$-function $f$ from its noisy measurements taken with respect to the first $q$ elements of a Riesz basis $\{\psi_k\}_{k\in\mathbb{N}}$ in $\mathcal{L}_2(D;\mathbb{C})$, 
 by leveraging the statistical information contained in the noisy coefficients of a training set $\{f_i\}_{i=1}^n$ with respect to the first $p$ elements of another orthonormal basis $\{\varphi_{\ell}\}_{\ell\in\mathbb{N}}$ in $\mathcal{L}_2(D;\mathbb{C})$. The main steps of the reconstruction method are summarized in Algorithm \ref{Algo:gs-fpca}.  If step {\footnotesize\textbf 1} of Algorithm \ref{Algo:gs-fpca} is computed using classical PCA and step {\footnotesize\textbf 3} is computed using least-squares, then the resulting procedure corresponds to the algorithm theoretically analyzed in this section. We note however, that classical PCA can be replaced by sparse PCA, while least-squares can be regularized by an $\ell_2$-term, as outlined in Sections~\ref{sub:spca} and \ref{sub:rr}, respectively.

\begin{algorithm}[htbp]
\SetAlgoLined
\IncMargin{1em}
\DontPrintSemicolon
\KwIn{$\{\varphi_{\ell}\}_{\ell=1}^p$, $\{\psi_k\}_{k=1}^q$, 
$\{b_k:=\langle f,\psi_k \rangle + w_k \}_{k=1}^q\subseteq \mathbb{C}$, 
$\{y_i := (\langle f_i, \varphi_{1} \rangle + z_{i 1},\ldots,\langle f_i, \varphi_{p} \rangle + z_{i p})^{\top}\}_{i=1}^n \subseteq \mathbb{C}^p$, $m\in\mathbb{N}$, $m<p$
}
Compute first $m$ eigenvectors $\{\hat e^p_j\}_{j=1}^m$ of $n^{-1}\sum_{i=1}^n (y_i-\hat\mu_{Y})\overline{(y_i-\hat\mu_{Y})}$, $\hat\mu_{Y}:=n^{-1}\sum_{i=1}^n y_i$.\;
Compute coefficients $\langle \hat\phi^p_j, \psi_k \rangle$ and $ \langle \hat\mu_p,\psi_k \rangle$, where $\hat\phi_j^p:=(\varphi_1,\ldots,\varphi_p)\hat e_j^p$, $\hat\mu_p:=(\varphi_1,\ldots,\varphi_p)\hat\mu_Y$.\;
Compute $\{\hat\alpha_j\}_{j=1}^m$ by solving for $\{\alpha_j\}_{j=1}^m$ in 
\begin{equation*}
\begin{pmatrix} \langle \hat\phi^p_1, \psi_1 \rangle & \cdots & \langle \hat\phi^p_m, \psi_1 \rangle  \\ \vdots & & \vdots \\  \langle \hat\phi^p_1, \psi_q \rangle & \cdots & \langle \hat\phi^p_m, \psi_q \rangle  \end{pmatrix} \begin{pmatrix} \alpha_1\\ \vdots\\ \alpha_m \end{pmatrix} = \begin{pmatrix} b_1 - \langle \hat\mu_p,\psi_1 \rangle\\ \vdots \\ b_q - \langle \hat\mu_p,\psi_q \rangle \end{pmatrix}.
\end{equation*}\;
\KwOut{$\hat f_{\mathrm{GS\text{-}FPCA}} := \hat\mu_p + \sum_{j=1}^m \hat\alpha_j \hat \phi_j^p$.}
\caption{Pseudo-code of the GS-FPCA algorithm}
\label{Algo:gs-fpca}
\end{algorithm}
 
In what follows, we provide a theoretical analysis of GS-FPCA reconstuction for any realization of a $\mathcal{L}_2(D;\mathbb{C})$-valued random field $F$ and a random sample $F_1,\ldots,F_n$. To this end, we consider the following random variables:
\begin{itemize}
\item[(i)]  $\{\langle F,\psi_k \rangle + W_k \}_{k=1}^q$, where  $W_1,\ldots,W_q$ are i.i.d.~Gaussian in $\mathbb{C}$ with mean zero and variance $\sigma^2$, which are also independent of $F_1,\ldots,F_n$. In particular, these random variables yield a realization of the low-resolution measurements in \eqref{eq:observations_Fourier}.
\item[(ii)] $\{Y_i :=Y_i(F_i,Z_i)= (\langle F_i, \varphi_{1} \rangle + Z_{i 1},\ldots,\langle F_i, \varphi_{p} \rangle + Z_{i p})^{\top}\}_{i=1}^n$, where $\{Z_i:= (Z_{i 1},\ldots, Z_{i p})^{\top}\}_{i=1}^n$ are i.i.d.~Gaussian in $\mathbb{C}^p$ with mean zero and covariance $\tilde\sigma^2I_p$, which are also independent of $F_i$'s and $W_k$'s.  As discussed previously, we model our high-resolution training  set \eqref{eq:observations_coef} as a realization of such random variables.
\end{itemize}
We define the reconstruction space as
\begin{equation}\label{eq:recspace}
\hat{\mathcal{E}}_m^p:=\text{span}\{\hat\phi_1^p,\ldots,\hat\phi_m^p\}, 
\end{equation}
where $\hat\phi_j^p:=(\varphi_1,\ldots,\varphi_p)\hat e_j^p$  and $\hat e^p_j$ is defined as the $j$th eigenvector of the sample covariance  $\hat{\Sigma}_{Y}:=n^{-1}\sum_{i=1}^n (Y_i-\hat\mu_Y)\overline{(Y_i-\hat\mu_Y)}$, where $\hat\mu_Y:=n^{-1}\sum_{i=1}^n Y_i$ denotes the sample mean. Namely, 
\begin{equation}\label{eq:eigenvec}
\hat e^p_j := \argmax_{\{v\in\mathbb{C}^p:\|v\|_2=1\}} \bar{v} \hat\Sigma_{Y}^{(j-1)} v, \quad j=1,\ldots,p,
\end{equation}
where $\hat\Sigma_{Y}^{(j)} := (I_p -\hat e^p_j\overline{\hat e^p_j}) \hat\Sigma_{Y}^{(j-1)} (I_p -\hat e^p_j\overline{\hat e^p_j} )$ and $\hat\Sigma_{Y}^{(0)}:=\hat\Sigma_{Y}$. 
Writing $\hat\mu_p:=(\varphi_1,\ldots,\varphi_p)\hat\mu_Y$, we now propose to estimate $F-\hat\mu_p$ in the reconstruction space $\hat{\mathcal{E}}_m^p$ defined in \eqref{eq:recspace}.
Specifically, we define the estimator of $F$ as 
\begin{equation}\label{eq:GSFPCA}
\hat F_{\mathrm{GS\text{-}FPCA}} := \hat\mu_p + \sum_{j=1}^m \hat\alpha_j \hat \phi_j^p,
\end{equation}
where the coefficients $\{\hat\alpha_j\}_{j=1}^{m}$ are the least-square solution to the linear system
\begin{equation}\label{eq:GSFPCAsystem}
\begin{pmatrix} \langle \hat\phi^p_1, \psi_1 \rangle & \cdots & \langle \hat\phi^p_m, \psi_1 \rangle  \\ \vdots & & \vdots \\  \langle \hat\phi^p_1, \psi_q \rangle & \cdots & \langle \hat\phi^p_m, \psi_q \rangle  \end{pmatrix} \begin{pmatrix} \alpha_1\\ \vdots\\ \alpha_m \end{pmatrix} = \begin{pmatrix} \langle F,\psi_1 \rangle + W_1 \\ \vdots \\ \langle F,\psi_q \rangle + W_q   \end{pmatrix} -
\begin{pmatrix}  \langle \hat\mu_p,\psi_1 \rangle \\ \vdots \\  \langle \hat\mu_p,\psi_q \rangle \end{pmatrix},
\end{equation}
with respect to $\{\alpha_j\}_{j=1}^m\in\mathbb{C}^m$, 
namely 
\begin{equation}\label{eq:hatalpha}
\{\hat\alpha_j\}_{j=1}^{m} := \argmin_{\{\alpha_j\}_{j=1}^m\subseteq\mathbb{C}} \sum_{k=1}^q \bigl| \langle F,\psi_k \rangle +W_k -\langle \hat\mu_p,\psi_k \rangle -  \sum_{j=1}^m \alpha_j  \langle\hat\phi^p_j,\psi_k\rangle \bigr|^2.
\end{equation}
It is useful to note that, if we denote the random system matrix in \eqref{eq:GSFPCAsystem} by $\hat{A}_{m,q}$, which takes values in $\mathbb{C}^{q\times m}$,
and the system matrix in \eqref{eq:GSsystem} by $A_{p,q}\in\mathbb{C}^{q\times p}$,  
since $ \langle\hat\phi^p_j, \psi_k \rangle=\sum_{\ell=1}^p  (\hat{e}_j^p)^{(\ell)}  \langle \varphi_{\ell}, \psi_k \rangle$ and $\langle\hat\mu_p, \psi_k \rangle=\sum_{\ell=1}^p  (\hat\mu_Y)^{(\ell)}  \langle \varphi_{\ell}, \psi_k \rangle$, we have $\hat{A}_{m,q}=A_{p,q}( \hat e^p_1, \ldots, \hat e^p_m )$.

When compared to the GS-reconstruction $\hat{f}_{\mathrm{GS}}$ defined in \eqref{eq:GS}, our reconstruction $\hat{F}_{\mathrm{GS\text{-}FPCA}}$ defined in \eqref{eq:GSFPCA} also takes values in  $\mathcal{G}_p$, but now the well-posedness of our solution depends on the value of random variable $\cos\angle(\hat{\mathcal{E}}^p_m,\mathcal{F}_q)\propto\sigma_{\min}(\hat{A}_{m,q})$ instead of $\cos\angle(\mathcal{G}_p,\mathcal{F}_q)\propto\sigma_{\min}(A_{p,q})$. For a sufficiently large $n$, we can show that our proposed estimator
can  stably achieve the $\mathcal{G}_p$-rate of approximation provided $m$ and $q$ are such that $\cos\angle(\mathcal{E}_m,\mathcal{F}_q)$ is bounded away from zero and $\| Q_{\mathcal{E}_m^{\perp}}  (F-\mu) \|$ is sufficiently small. Specifically, we can show the following.

\begin{theorem}\label{thm:main}
Consider the setting of Lemma~$\ref{lem:eig_mean_est}$ and let $m,q,p,n$ and $\delta$ be  such that $\sin\angle(\mathcal{E}_m,\mathcal{F}_q)<1-\tilde\epsilon_{mpn\delta}$, where $\mathcal{F}_q:=\text{span}\{\psi_1,\ldots,\psi_q\}$
and $\{\psi_k\}_{k\in\mathbb{N}}$ is such that \eqref{eq:rieszbounds} holds. Then there exists $p_0$ such that for any $p\geq \max\{p_0,m\}$, $n\geq p$ and any $\mathcal{L}_2(D;\mathbb{C})$-valued random field $F$,  estimator $\hat{F}_{\mathrm{GS\text{-}FPCA}}$ defined in \eqref{eq:GSFPCA}  satisfies
\begin{align*}
\|\hat{F}_{\mathrm{GS\text{-}FPCA}}-F\|\leq \frac{ \| Q_{\mathcal{E}_m^{\perp}}  (F-\mu)  \| +\tilde\epsilon_{mpn\delta}\|F-\mu\| +\bar\epsilon_{np\delta'} +\sigma(2m+2\log(1/\delta''))^{1/2}(qr_1)^{-1/2}}{{1 - \sin\angle(\mathcal{E}_m,\mathcal{F}_q) - \tilde\epsilon_{mpn\delta}}}
\end{align*}
with probability  at least $1-\delta-\delta'-\delta''$.
\end{theorem}

The proof of this theorem is given in Appendix~\ref{app:proofs}, while here we discuss its consequences.
First recall that due to \eqref{PCAorder}, if the approximation rate in $\mathcal{G}_p$ is of order $p^{-\gamma}$ and provided $n\gtrsim p^{2\gamma+1}$, $\delta,\delta'\geq 2e^{-p/2}$ and $\lambda_{m}-\lambda_{m+1}$ is lower bounded, then $\max\{\tilde\epsilon_{mpn\delta},\bar\epsilon_{pn\delta'}\} = \mathcal{O} \bigl( \sqrt{m} / p^{\gamma} \bigr)$. Thus,  if also $\delta''\geq e^{-m}$, under the conditions of Theorem~\ref{thm:main}---namely, for a fixed $m$ and constant $\epsilon\gtrsim  \sqrt{m} / p_0^{\gamma}$, if $q=q(m)$ is sufficiently large so that $\cos\angle(\mathcal{E}_m,\mathcal{F}_q) > \sqrt{\epsilon}$---then, with probability at least $1-\delta-\delta'-\delta''$, we have
\begin{equation*}
\|\hat{F}_{\mathrm{GS\text{-}FPCA}}-F\| = \mathcal{O} \Bigl( \| Q_{\mathcal{E}_m^{\perp}}  (F-\mu)  \| + (\|F-\mu\|+1)   \sqrt{m} / p^{\gamma}  +  \sigma\sqrt{m/q} \Bigr).
\end{equation*}
We note that this result holds for any $\mathcal{L}_2(D;\mathbb{C})$-valued random field $F$, however if $F\sim P$ and $F$ is independent of $F_1,\ldots,F_n,$ as well as of the noise variables $Z_i$'s and $W_k$'s, then we can bound the expectation of the right-hand side by using the KL expansion of $F$ \eqref{eq:KLexp}. In particular, similarly as in Section~\ref{sub:gs-random}, by introducing the truncation operator \eqref{trunc} and considering probability measure $P$ on the space of uniformly $\tau$-bounded functions in $\mathcal{L}_2(D;\mathbb{C})$, 
if  $n\gtrsim p^{2\gamma+1}$ and $p$ and $m$ are sufficiently large so that
$e^{-p/2}+e^{-m}\lesssim \bigl\{\bigl(\sum_{j>m}\lambda_j\bigr)^{1/2} + \sqrt{m} / p^{\gamma}  +   \sigma\sqrt{m/q} \bigr\}/\tau\Delta$, 
under the conditions of Theorem~\ref{thm:main} 
and provided $F\sim P$, $F$  independent of all other random variables, we have
\begin{equation}\label{eq:bound_thm}
\mathbb{E}\|T_{\tau}(\hat{F}_{\mathrm{GS\text{-}FPCA}})-F\| = \mathcal{O} \Bigl(
\bigl(\textstyle\sum_{j>m}\lambda_j\bigr)^{1/2}  + \sqrt{m} / p^{\gamma} +  \sigma\sqrt{m/q} \Bigr),
\end{equation}
where we used $\mathbb{E}\|F-\mu\|\leq \bigl(\int_{D}K(u,u)\D u \bigr)^{1/2} =
 \bigl(\sum_{j\in\mathbb{N}}\lambda_j \bigr)^{1/2} $, which we regard as a constant, and also $\mathbb{E}\| Q_{\mathcal{E}_m^{\perp}}  (F-\mu)  \|\leq\bigl(\sum_{j>m}\lambda_j \bigr)^{1/2} $, which holds due to Jensen's inequality and \eqref{eq:KLexp} and corresponds to the optimal expression in \eqref{eq:KLopt}. In particular, if probability measure $P$ is strictly low-rank, then there exists $m_0$ such that for all $m\geq m_0$, $\sum_{j>m}\lambda_j$ is zero. 
However, it is enough for the eigenvalues $\{\lambda_j\}_{j\in\mathbb{N}}$ to decrease relatively quickly, for this term to become sufficiently small. 

It is now instructive to compare the  rate of estimation in \eqref{eq:bound_thm} with  the rate of estimation of the GS-reconstruction from \eqref{GSorder} that has order $1/ p^{\gamma}  + \sigma\sqrt{p/q}$ provided $\cos\angle(\mathcal{G}_p,\mathcal{F}_q)$ is bounded away from zero. Note that if there exists $m_0$ such that for all $m\geq m_0$, $\sum_{j>m}\lambda_j \lesssim m / p^{2\gamma} + \sigma^2m/q$, then  the  GS-FPCA rate of estimation is of order  $\sqrt{m} / p^{\gamma}  + \sigma\sqrt{m/q}$, provided $\cos\angle(\mathcal{E}_m,\mathcal{F}_q)$ is bounded away from zero. Remarkably, in the noiseless case when   $\sigma=0$, for a fixed $m\geq m_0$ and increasing $p$, the resolution of the  GS-FPCA reconstruction increases as $p^{\gamma}$, only at the cost of increasing  the number of training observations $n=n(p)$, since the number of measurements $q$ does not exhibit dependence on $p$ (for sufficiently large $p,n$ and $q$). 
In contrast, for  the GS-reconstruction to achieve the same resolution we need to increase the number of measurements $q=q(p)$ so that $\cos\angle(\mathcal{G}_p,\mathcal{F}_q)$ remains bounded away from zero. Moreover, in noisy case when $\sigma>0$, for GS reconstruction, $q$ needs to increase with respect to $\sigma^2 p$, while for GS-FPCA, it needs to increase with respect to $\sigma^2 m$, which may present a considerable improvement in case of $p \gg m$. 

It is also instructive to compare the GS and GS-FPCA reconstructions from a computational-complexity point of view. The computational complexity of GS, that is, the computational complexity of solving system \eqref{eq:GSsystem}, is of order $qp$, whereas the computational complexity of deploying GS-FPCA, that is, the computational complexity of solving system \eqref{eq:GSFPCAsystem}, is of order $qm$, which is less or equal to that of GS since $p>m$. In the Fourier-wavelet case, due to fast Fourier and wavelet transform algorithms, the complexity order of GS can be reduced to $q\log p$ \cite{Gataric2016}, which is still slower than GS-FPCA if $\log p \gtrsim  m$.

The asymptotic bound \eqref{eq:bound_thm} is further illustrated by numerical examples of Section~\ref{sub:1D}, where a low rank 1D model is used with $m$ such that $\mathbb{E}\| Q_{\mathcal{E}_m^{\perp}}  (F-\mu)  \| = 0 $ and $q$ such that $\cos\angle(\mathcal{E}_m,\mathcal{F}_q)$ is bounded away from zero, thus satisfying the conditions required for \eqref{eq:bound_thm} to hold. Moreover, in Section~\ref{sub:2D}, we illustrate the performance of the proposed reconstruction using a more realistic 2D model, where in Figure~\ref{fig:sing_val_phantom} we vary $m$ for a fixed $q$ and show that it is possible to choose $m$ such that $\mathbb{E}\| Q_{\mathcal{E}_m^{\perp}}  (F-\mu)  \|$ is small and such that $\cos\angle(\mathcal{E}_m,\mathcal{F}_q)$ is bounded away from zero, as required by Theorem~\ref{thm:main} for the high-resolution rate of approximation.


\subsection{GS-FPCA with sparse principal components}\label{sub:spca}

Under the assumption that the functional principal components are sparse with respect to the reconstruction basis, one can use sparse PCA instead of classical PCA to estimate the PCs in \eqref{eq:recspace} and thus reduce the required size of the training set $n$. The sparsity assumption is commonly leveraged when reconstructing a signal of interest, as it is know that natural images are sparse with respect to wavelets.
Within the GS-FPCA framework, the sparsity assumption on FPCs implies that $\{\phi^p_j\}_{j=1}^m$ are sparse with respect to $\{\varphi_{\ell}\}_{\ell\in\mathbb{N}}$ for a sufficiently large $p$. Thus, it is reasonable to assume that only $k<p$ entries of $e^p_j\in\mathbb{C}^p$ are different than zero, in which case, we can use sparse PCA to compute $\{\hat{e}_j^p\}_{j=1}^m$ by constraining the optimizer in \eqref{eq:eigenvec} to be sparse. In particular, writing $\mathrm{nnzr}(v)$ for  the number of non-zero rows of a vector $v\in\mathbb{C}^p$, sparse PCA computes the first PC by solving
$$
\hat{e}_1^p := \argmax_{\{v:\mathrm{nnzr}(v)\leq k, \|v\|_2=1\}} \bar{v}\hat\Sigma_{Y} v,
$$
whereas higher-order PCs can be computed via a modified deflation scheme or by maximizing the trace of $V^*\hat\Sigma_{Y} V$  over orthonormal matrices $V$ such that $\mathrm{nnzr}(V)\leq k$, see e.g.~\cite{Gataric2020}.  

There are many existing algorithms for computing sparse principal components, see for example~\cite{Zou2006, dAspremont2007, Ma2013, Gataric2020}, and also, 
statistical and computational properties of sparse PCA are quite well understood due to the work by~\cite{Johnstone2009,Vu2013,WBS2016} and others. In particular, due to these results, we know that by using sparse PCA we can readily reduce the term $\sqrt{mp/n}$ in \eqref{PCAorder} to $\sqrt{mk\log{p}/n}$, and therefore reduce the number of observations $n$ required by Theorem~\ref{thm:main}. In Section \ref{sec:sim}, we examine both classical and sparse PCA when computing the GS-FPCA reconstruction in our numerical simulations and indeed observe a regularization effect due to sparse PCA in a high-dimensional setting when $n$ is small compared to $p$.

\subsection{GS-FPCA with $\ell_2$-regularization}\label{sub:rr}

For an improved performance in a noisy setting, one may want to add $\ell_2$-regularization the least-squares estimation of the coefficients $\{\alpha_j\}_{j=1}^m$ in \eqref{eq:hatalpha}. 
From the KL expansion of $Q_{\mathcal{G}_p}F$, we know that $(\langle F-\mu_p,\phi_1^p \rangle, \ldots, \langle F-\mu_p,\phi_m^p \rangle)^{\top}$ has mean zero and covariance $\Lambda_m:=\mathrm{diag}({\lambda}_1^p,\ldots,{\lambda}_m^p)$, and thus, if $P$ is Gaussian, it is reasonable to impose prior distribution $ N_m(0,\hat{\Lambda}_m)$ on these coefficients and use the corresponding MAP estimator instead of the ML estimator. This leads to a ridge regression problem where a weighted $\ell_2$-regularization term is added to the least-squares objective function, so that instead of \eqref{eq:hatalpha} we have
$$
\{\hat\alpha_j\}_{j=1}^{m} := \argmin_{\{\alpha_j\}_{j=1}^m\in\mathbb{C}^m} \sum_{k=1}^q \bigl| \langle F,\psi_k \rangle +W_k -\langle \hat\mu_p,\psi_k \rangle -  \sum_{j=1}^m \alpha_j  \langle\hat\phi^p_j,\psi_k\rangle \bigr|^2
+ \lambda\sum_{j=1}^m (\hat{\lambda}_j^p)^{-1}|\alpha_j|^2,
$$ 
for some regularization parameter $\lambda>0$. 

Theoretical analysis of such regularization procedure would require a different approach to the one taken in this paper, however, due to classical results on ridge regression and Tikhonov regularization, see e.g.~\cite{Hsu2012b,Schonlieb2019},  
in this case  we expect a more robust estimation for a smaller $q$ relative to $m$. In particular, we expect a relaxed version of the condition with respect to $\cos\angle(\mathcal{E}_m,\mathcal{F}_q)$, since
the minimal singular value  of the regularized system matrix  is equal to the square root of $\lambda_m(\hat A_{m,q}^*\hat A_{m,q}+ \lambda\mathrm{diag}(\hat{\lambda}^p_1,\ldots\hat{\lambda}^p_m)^{-1})$,  which is lower-bounded by the square root of $\lambda_m(\hat A_{m,q}^*\hat A_{m,q})+\lambda/\hat{\lambda}^p_1$, due to Weyl's inequality \cite{Weyl1912}. However, this would come
at the price of a lower estimation rate that includes the order of $\sqrt{\lambda}$ even in the noiseless case where $\sigma=0$.  In our numerical results below, we also include such regularized estimation procedure, which in a noisy setting can further improve reconstruction performance.  

\section{Numerical simulations}\label{sec:sim}

\subsection{Examples with one-dimensional generative model}\label{sub:1D}

For numerical examples in this subsection, we simulate data  using the following generative model: 
\begin{equation}\label{sim_model}
f_i(u)=\sum_{j=1}^{m_0}\sqrt{\lambda_j} \xi_{ij} \phi_j(u),  \quad u\in D:=[0,1], \quad i=1,\ldots,n,
\end{equation}
where $\xi_{ij}$ are i.i.d.~standard normal random factors in $\mathbb{R}$, $\lambda_j:=m_0-j+1$ and  $\{\phi_j\}_{j=1}^{m_0}$ are FPCs with each $\phi_j(u)$ constructed as a linear combination of exponentials $\exp(-(u-u_0)^2/s_0)$ for various choices of $s_0$ and $u_0\in D$. Such $\{\phi_j\}_{j=1}^{m_0}$, for $m_0=10$, are shown in  Figure~\ref{fig:components} and several  $f_i$'s generated from this model are shown in Figure~\ref{fig:fi}. 

\begin{figure}[htb]
 \centering
 \includegraphics[scale=0.95]{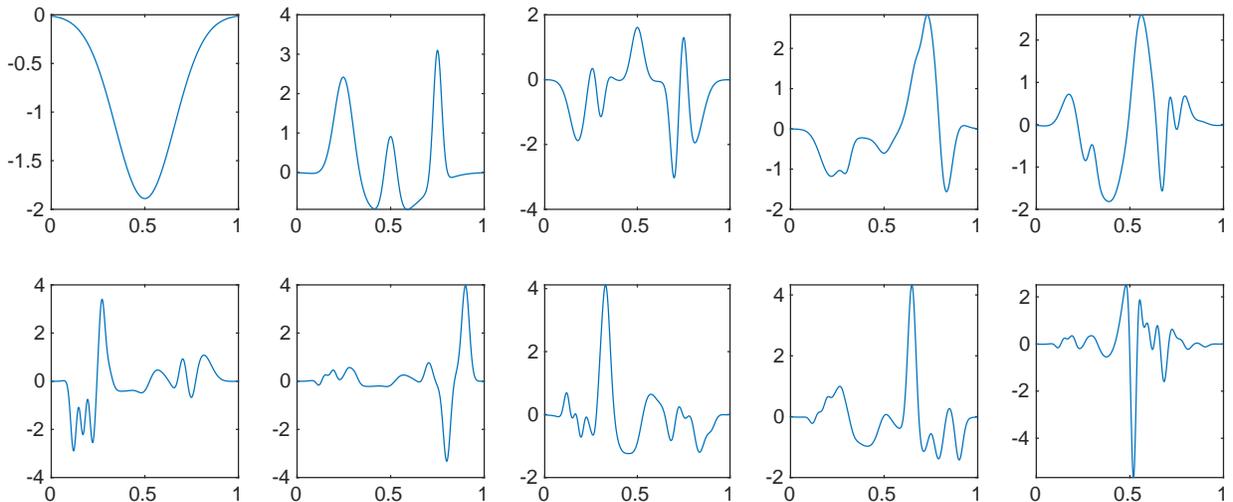}
 \caption{Functional principal components $\{\phi_j\}_{j=1}^{10}$ (ordered from left to right, top to bottom).}
 \label{fig:components}
\end{figure}

To construct our training set \eqref{eq:observations_coef}, we choose $\mathcal{G}_p$ as the span of the first $p$ elements of the boundary-corrected Daubechies wavelets of order $s$, introduced in \cite{CDV1993}, which constitute an orthonormal basis $\{\varphi_{\ell}\}_{\ell\in\mathbb{N}}$ in $\mathcal{L}([0,1];\mathbb{R})$, and compute the noisy high-resolution training observations as
\begin{equation}\label{wavelet_coef}
y_{i}:=\bigl( \langle f_i,\varphi_{1} \rangle, \ldots, \langle f_i,\varphi_{p} \rangle \bigr)^{\top}+z_i,\quad  i=1,\ldots,n,
\end{equation}
where each variable of the noise vector $z_i\in\mathbb{R}^p$ is generated from normal distribution $N(0,\tilde\sigma^2)$. For computation of appropriate wavelet functions we used Wavelab\footnote{Available at \href{http://www-stat.stanford.edu/~wavelab/}{www-stat.stanford.edu/$\sim$wavelab/}.} as well as the Matlab files from the Supplementary material of \cite{Gataric2016} for handling 2D boundary corrected wavelets and orders $s>3$. 
Finally, we simulate measurements \eqref{eq:observations_Fourier} by generating a new unseen observation $f$ from model \eqref{sim_model}, and computing its $q$ noisy Fourier samples as
\begin{equation}\label{fourier_coef}
\langle f, \psi_k \rangle + w_k := \int_0^1 f(u)\exp\bigl\{-\mathrm{i} 2\pi \bigl(k-\lfloor q/2\rfloor \bigr) u\bigr\}\D u + w_k, \quad k=1,\ldots,q,
\end{equation}
where noise $w_k\in\mathbb{C}$ is such that both $\mathrm{Re}(w_k)$ and $\mathrm{Im}(w_k)$ are from  $N(0,\sigma^2/2)$. It is important to note that $q/2$ is therefore the highest measured frequency. We also note that, when computing an infinite-dimensional inner-product, we discretize $D$ with increments $\Delta u$ so that $1/\Delta u \gg p$.

\begin{figure}[htb]
 \centering
 \includegraphics[scale=0.95]{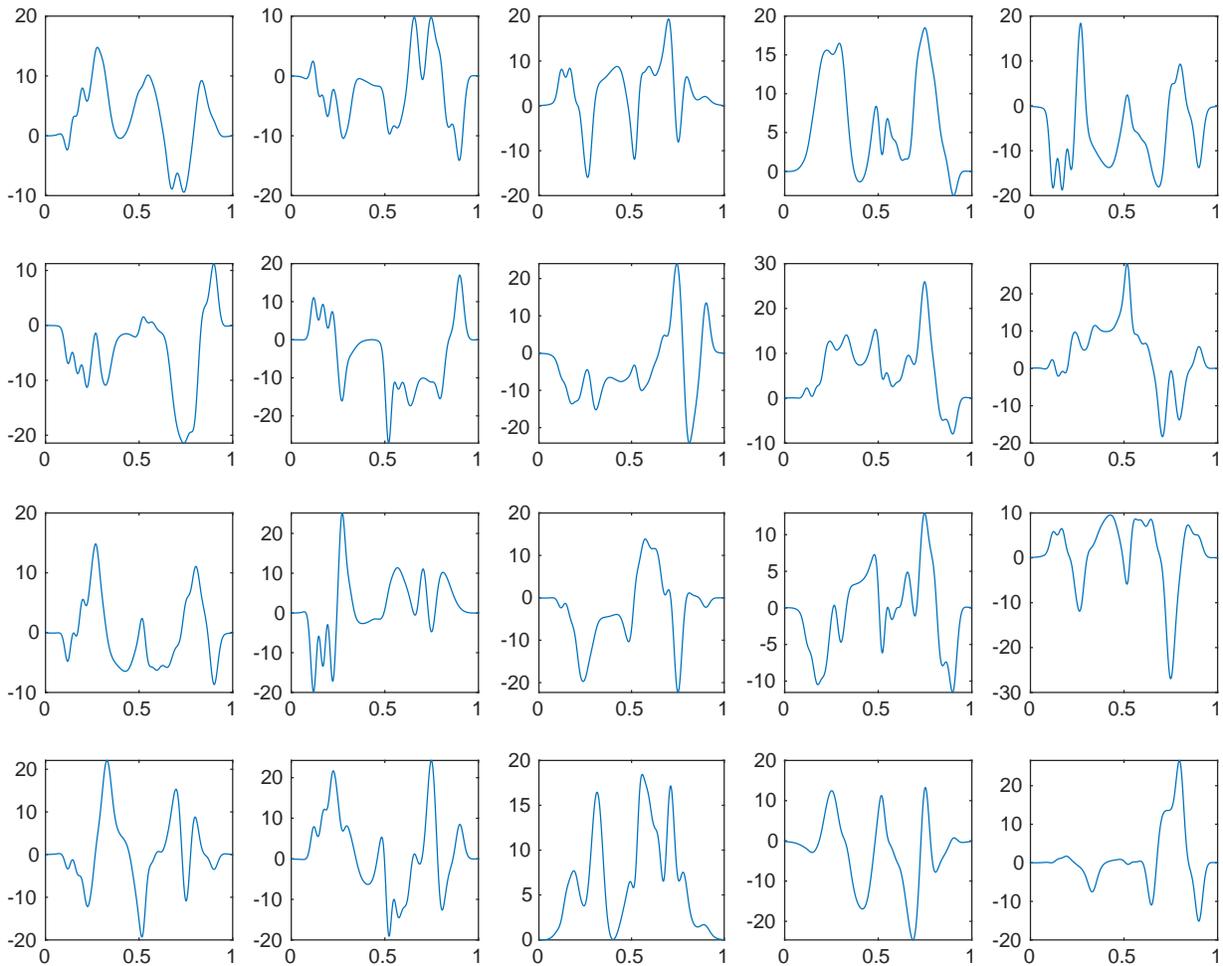}
 \caption{Several $f_i$'s from model \eqref{sim_model} with the principal components shown in Figure \ref{fig:components}. Such $f_i$'s are used to generate (noisy) training set \eqref{wavelet_coef} for the experiments in Figures \ref{fig:diff_wave}--\ref{fig:noiseless}.}
 \label{fig:fi}
\end{figure}

In the examples of this subsection, we assume that we have correctly specified $m$ so that $\hat{\mathcal{E}}_m^p$ has the dimension corresponding to the true rank $m_0$ of the model \eqref{sim_model} used to generate the data, and thus $\|Q_{\mathcal{E}_m^{\perp}}f\|=0$, for any $m\geq m_0$, which makes it  possible to better understand different terms in bound  \eqref{eq:bound_thm} that depend on problem parameters $q,n$ and $p$. Indeed, by inspecting the eigenvalues of $\hat\Sigma_Y$, it is very easy to correctly specify the true rank $m_0$ in this example even with very low SNR, so we leave the consideration of choosing appropriate $m$ for the next subsection.

In Figure~\ref{fig:diff_wave}, we first demonstrate  that using solely the low-frequency Fourier measurements \eqref{fourier_coef}, indeed it is impossible to accurately reconstruct ground truth $f$ directly in the high-resolution space $\mathcal{G}_p$  by the (regularized) GS reconstruction, whereas, if also accounting for the training observations \eqref{wavelet_coef}, then by using the GS-FPCA reconstruction as proposed in this paper it becomes possible to accurately reconstruct $f$ with high-resolution in $\mathcal{G}_p$.
Specifically, in this example, we use $\mathcal{G}_p$ with $p=128$ wavelets of order $s\in\{1,4,8\}$ and a relatively small number $q=12$ of Fourier measurements with noise $\sigma=0.02\sqrt{2}$. To compute $\hat{\mathcal{E}}^p_m$, we use $n=p$ training observations with noise $\tilde{\sigma}=0.01$.  
In Figure~\ref{fig:diff_wave}, in orange, black and green, we show $\hat f_{\mathrm{GS}}\in\mathcal{G}_p$ defined in \eqref{eq:GS}, whose coefficients are estimated either by the plain least-squares or with $\ell_2$ or $\ell_1$-regularization term $\lambda \sum_j^p|a_j|^2$ or $\lambda \sum_j^p|a_j|$, with $\lambda=0.04$; while in cyan and blue, we show  $\hat{f}_{\mathrm{GS}\text{-}\mathrm{FPCA}}\in\hat{\mathcal{E}}_m^p\subseteq \mathcal{G}_p$ defined in \eqref{eq:GSFPCA}, whose coefficients are computed either by the plain least-squares or its regularized version with additional term $\lambda\sum_{j=1}^m (\hat{\lambda}_j^p)^{-1}|\alpha_j|^2$ and parameter $\lambda=0.08$, as described in Section~\ref{sub:rr}.

\begin{figure}
 \centering
\begin{tabular}{ccc}
 $s=1$ & $s=4$ & $s=8$ \\
 \includegraphics[scale=0.74]{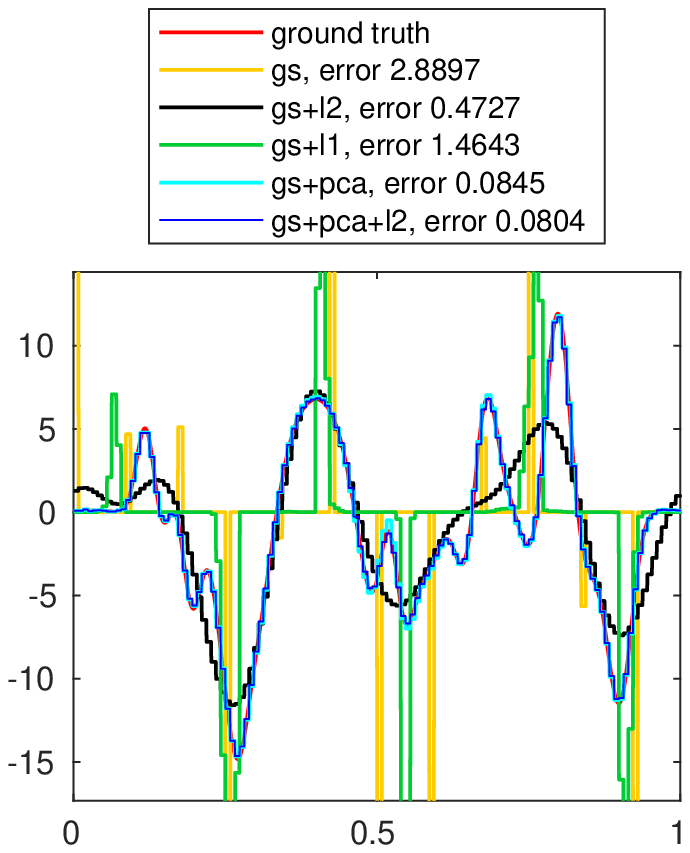} &
 \includegraphics[scale=0.74]{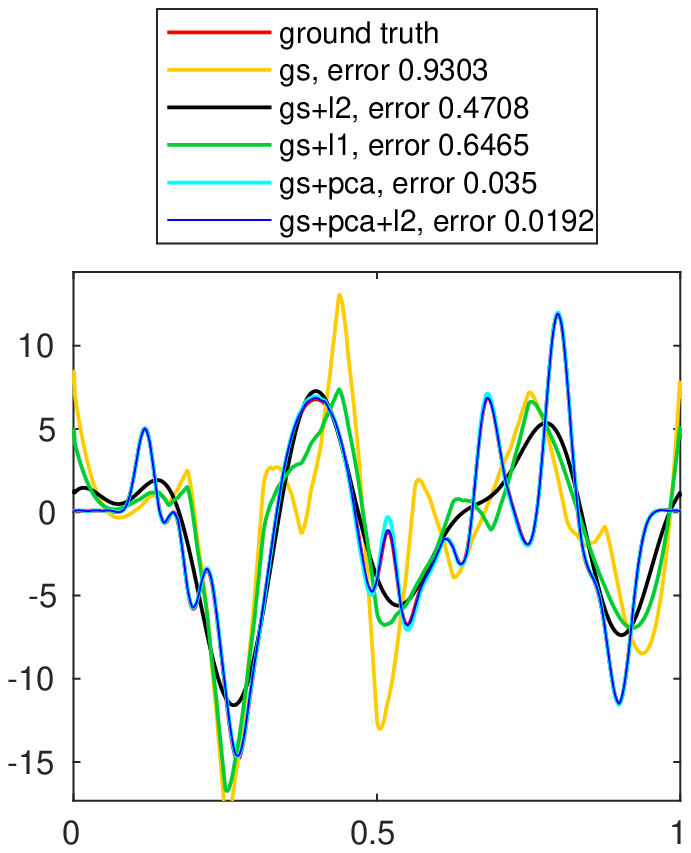} &
\includegraphics[scale=0.74]{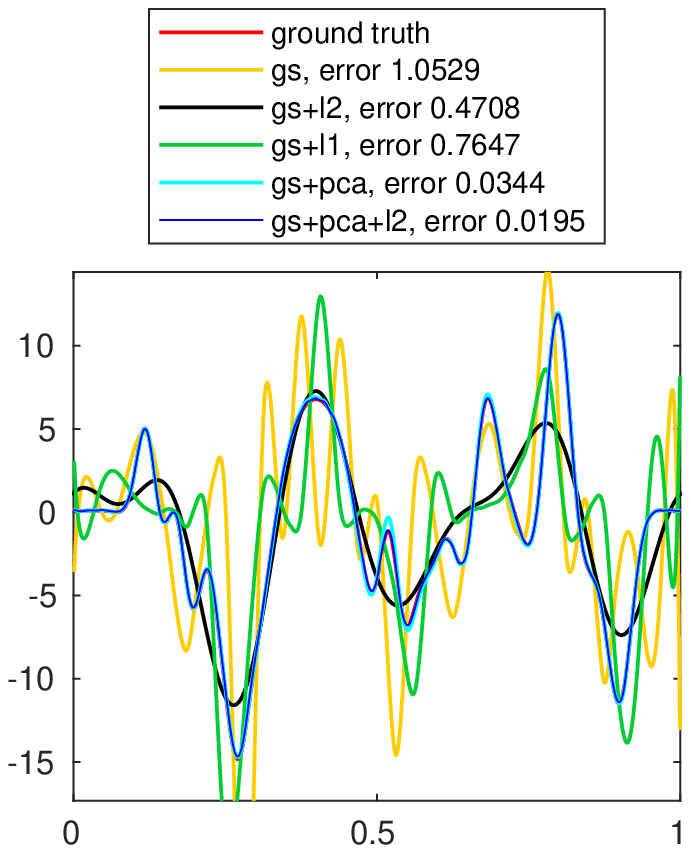}
\end{tabular}
 \caption{Reconstruction from $q=12$ noisy Fourier measurements \eqref{fourier_coef} and a training set of noisy wavelet coefficients  \eqref{wavelet_coef} with $n=p=128$. $\mathcal{G}_p$ consists of boundary-corrected Daubechies wavelets with $s\in\{1,4,8\}$ vanishing moments. 
We also report  relative reconstruction error $\|f-\hat{f}\|/\|f\|$.}
 \label{fig:diff_wave}
\end{figure}

Next, in Figure~\ref{fig:increasing_qnp} we inspect how the average relative error $\|f-\hat{f}\|/\|f\|$  behaves when varying problem parameters $q$, $n$ and $p$, in the noisy setting with $\sigma=0.02\sqrt{2}$ (SNR around 60 on average) and $\tilde{\sigma}=0.01$ (SNR around 40) and with Daubechies wavelets of order $s=4$. The average is computed over 30 repetitions of the experiment so that we reconstruct 30 different unseen $f$'s generated using the model in \eqref{sim_model}, while reconstruction is performed either by GS or GS-FPCA, where principal components are computed either by classical PCA or sparse PCA as discussed in Section~\ref{sub:spca}.  From the top-left panel of Figure~\ref{fig:increasing_qnp} we see that when (sparse) PCA is used to construct the reconstruction space, the error is on the order of the noise already for relatively small $q\geq 12$. We also note that in this noisy case when $q<12$, adding the $\ell_2$-regularization, as described in Section~\ref{sub:rr}, is helpful in increasing accuracy. On the other hand, much larger $q$ is needed to attain the same accuracy by using other variants of GS without the training set. From the top-right panel of Figure~\ref{fig:increasing_qnp}, we see that the desired accuracy is achieved already with $n\geq p/16$ in this example, and that for relatively small $n$ the accuracy is improved by using sparse PCA instead of classical PCA. From the bottom-left panel of Figure~\ref{fig:increasing_qnp} we see that by increasing $p$ and $n=2p$ we are indeed improving the resolution of our reconstruction (up to the order of the noise), even when $q$ fixed, confirming the conclusion of our theoretical results. Finally, in the bottom-right panel of  Figure~\ref{fig:increasing_qnp}, we vary the level of noise $\sigma^2$, so that SNR increases from around 4 to 240 on average, thus confirming that the error is a linear factor of noise when $\sigma \sqrt(m/q)$ is its driving term.

\begin{figure}
\centering
\begin{tabular}{cc}
$p=512$, $n=2p$ & $p=512$, $q=12$  \\
\includegraphics[scale=0.65]{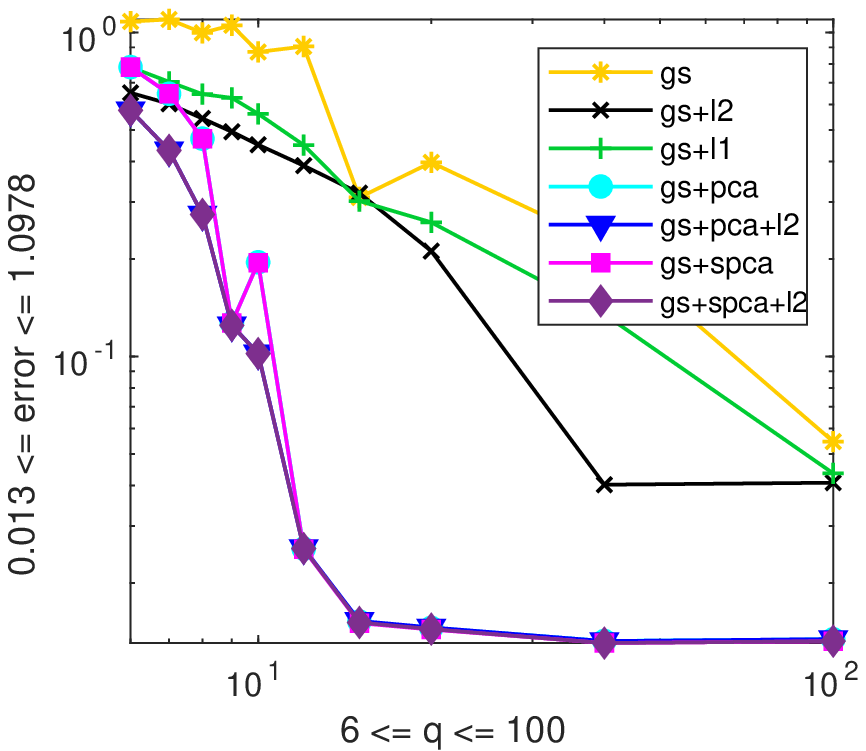} &
\includegraphics[scale=0.65]{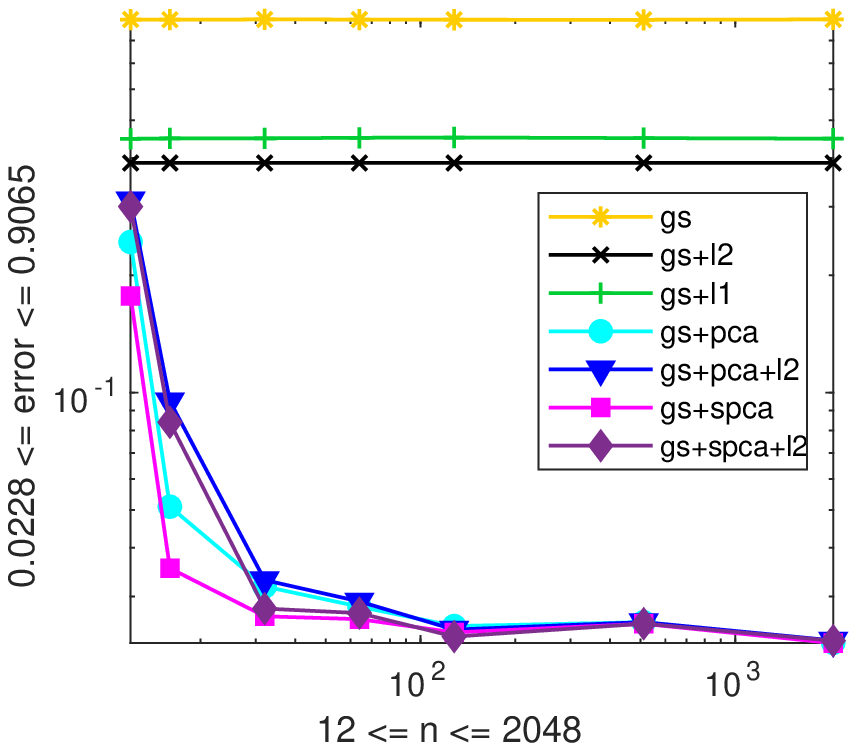} \\
$n=2p$, $q=12$ & $p=512$, $n=2p$, $q=12$ \\
\includegraphics[scale=0.65]{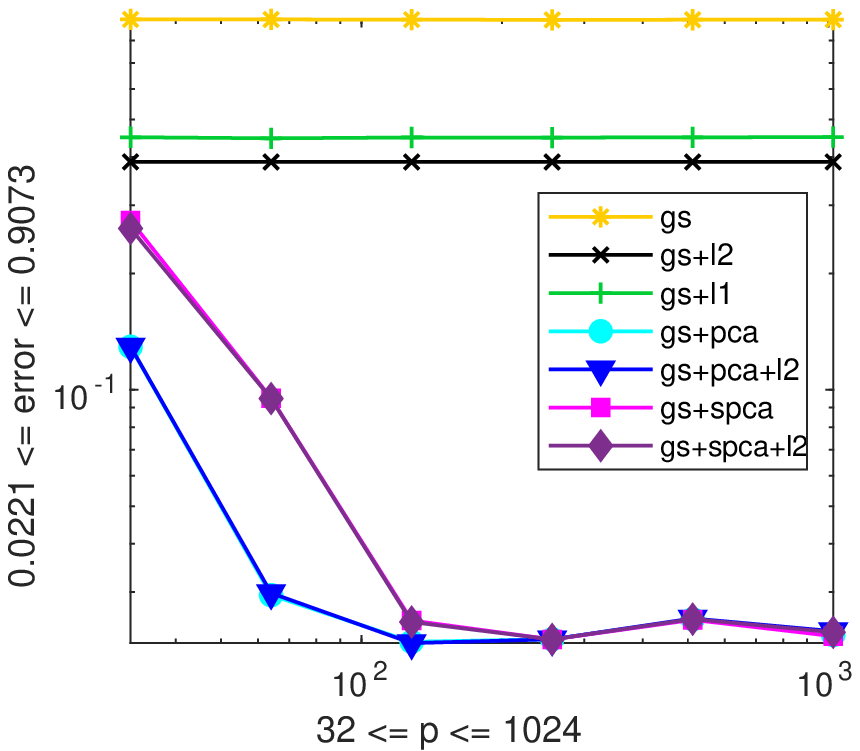} &
\includegraphics[scale=0.65]{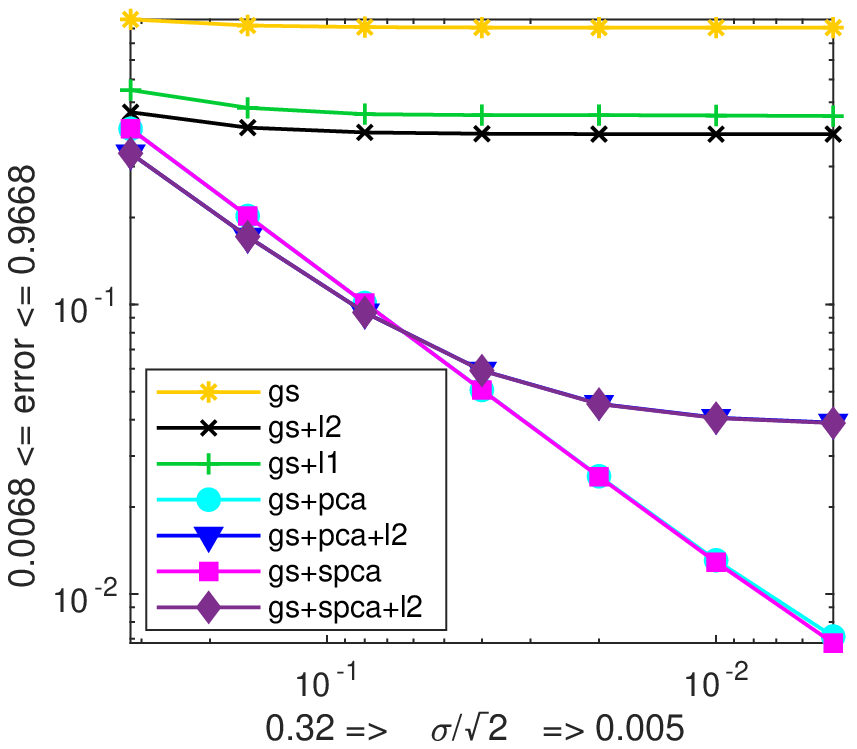} 
\end{tabular}
\caption{Average relative error $\|f-\hat{f}\|/\|f\|$ for varying $q$ (top-left), $n$ (top-right), $p$ (bottom-left) and $\sigma$ (bottom-right), in the noisy case where $\sigma=0.02\sqrt{2}$ (when fixed) and $\tilde{\sigma}=0.01$. An average over 30 repetitions is computed using different reconstruction methods  (different colors/markers). 
}
\label{fig:increasing_qnp}
\end{figure}

To further examine our theoretical results, we present in Figure~\ref{fig:noiseless} the noiseless case where we take $\sigma=\tilde\sigma=0$ and $n=2p$, and use different wavelet subspaces with varying number of vanishing moments $s\in\{1,2,4\}$. As depicted by our bound \eqref{eq:bound_thm} derived from Theorem \ref{thm:main}, in Figure~\ref{fig:noiseless}, we see that we can indeed attain the approximation rate associated to the $p$-dimensional space $\mathcal{G}_p$, which in the case of wavelets with $s$ moments corresponds to $p^{-\gamma}$, $\gamma <s$,  provided $f$ is $\gamma$-H\"{o}lder continuous. 
In fact, via GS-FPCA framework, we can attain such rate with relatively small $q$, while much larger $q$ is required when reconstructing directly in $\mathcal{G}_p$  via GS. 

\begin{figure}
\centering
\begin{tabular}{ccc}
 & $q=12$ & \hspace{-2.5cm} $p=512$ \\
 \hspace{1cm} & \includegraphics[scale=0.65]{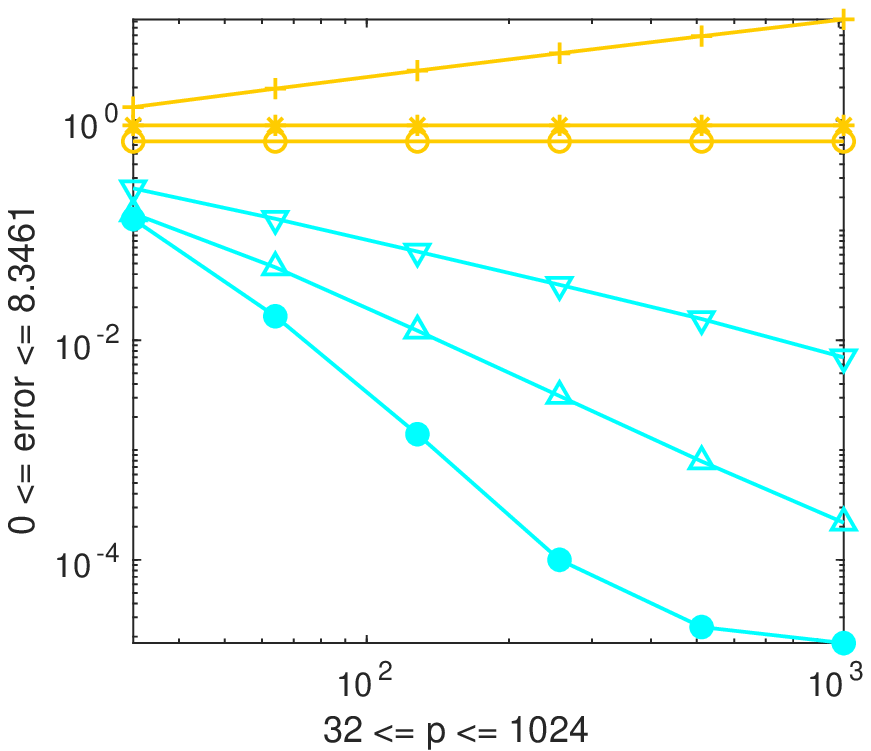} &
\includegraphics[scale=0.65]{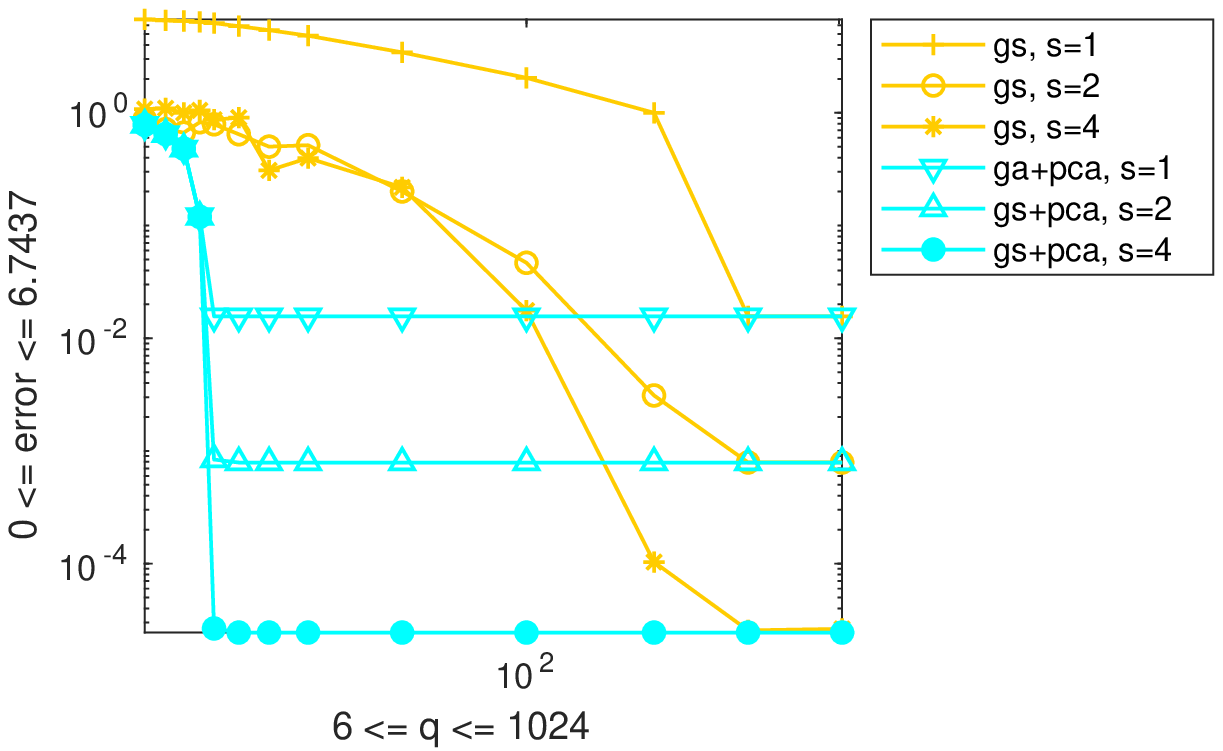}
\end{tabular}
\caption{The relative reconstruction error using different wavelet subspaces $\mathcal{G}_p$ with $s\in\{1,2,4\}$ vanishing moments in combination with either GS (orange) or GS-FPCA (cyan) for increasing $p$ and fixed $q$ (left) or for increasing $q$ and fixed $p$ (right) in the noiseless case when $n=2p$.}
\label{fig:noiseless}
\end{figure}

\subsection{Examples with two-dimensional Shepp--Logan phantom}\label{sub:2D}

In the following examples, we use 2D images of a Shepp--Logan (SL) phantom, which can be generated up to an arbitrary resolution $1/\Delta u$ by using Matlab's function `phantom($E$,$1/\sqrt{\Delta u}$)', where each row of matrix $E\in\mathbb{R}^{10\times6}$ specifies an ellipse in the image using $6$ different parameters and $\sqrt{\Delta u}\times \sqrt{\Delta u}$ specifies the discretization of the 2D domain $D:=[0,1]\times[0,1]$. Crucially, we choose $\Delta u$ so that $1/\Delta u\gg p, q$ and so that we can simulate an infinite-dimensional measurement model. In particular $1/\Delta u = 256^2$, $p=64^2$ and $q=32^2$ in all the examples of this subsection. 
Specifically, the measurements \eqref{eq:observations_Fourier} of an unseen phantom $f$ are computed by approximating the Fourier coefficients 
\begin{equation}\label{eq:2dFourier_coef}
\langle f, \psi_{k,j}\rangle:=\int_0^1\int_0^1 f(u,v)\exp\bigl\{-\mathrm{i} 2\pi [ (k-\lfloor\sqrt{q}/2\rfloor) u + (j-\lfloor\sqrt{q}/2\rfloor)]\bigr\} \D u \D v,
\end{equation}
with respect to the 2D Fourier basis yielding the span of $\mathcal{F}_q:=\mathrm{span}\{\psi_{k,j}, k, j=1,\ldots,\lfloor\sqrt{q}\rfloor\}$. In addition, we perturb both the real and imaginary part of the Fourier coefficients $b$ with the noise vector $w$ from $N_q(0,(0.0002)^2 I_q)$  so that the SNR measured as $\|b\|_2 /\|w\|_2$ is around $36$.

The training set \eqref{eq:observations_coef} is obtained by first generating $n=512$ phantoms $\{f_i\}_{i=1}^n$, where
each $f_i$ 
is computed by randomly perturbing  matrix $E_0$   used to compute the Matlab's default phantom, which can be retrieved in Matlab by executing~`$[\sim,E_0]$ = phantom()'.
Next, we compute  $y_i^{(\ell)}:=x_i^{(\ell)}+z_i^{(\ell)}$, where $p=4096$, $\ell=1,\ldots,p$,  $x_i^{(\ell)}:=\langle f_i, \varphi_{\ell}\rangle$ are the coefficients of $f_i$ with respect to the 2D boundary-corrected wavelets and noise $z_i^{(\ell)}$ is generated from zero-mean Gaussian with $\tilde{\sigma}=0.0001$, so that SNR measured as $\|x_i\|_2/\|z_i\|_2$ is around $36$ on average. In Figure \ref{fig:train_phantom} we show several such  training observations by displaying $\sum_{\ell=1}^p y_i^{(\ell)} \varphi_{\ell}(u)$, $u\in D$.

\begin{figure}
\centering
\includegraphics[scale=0.95,trim={2.2cm 1.5cm 1.5cm 1cm},clip]{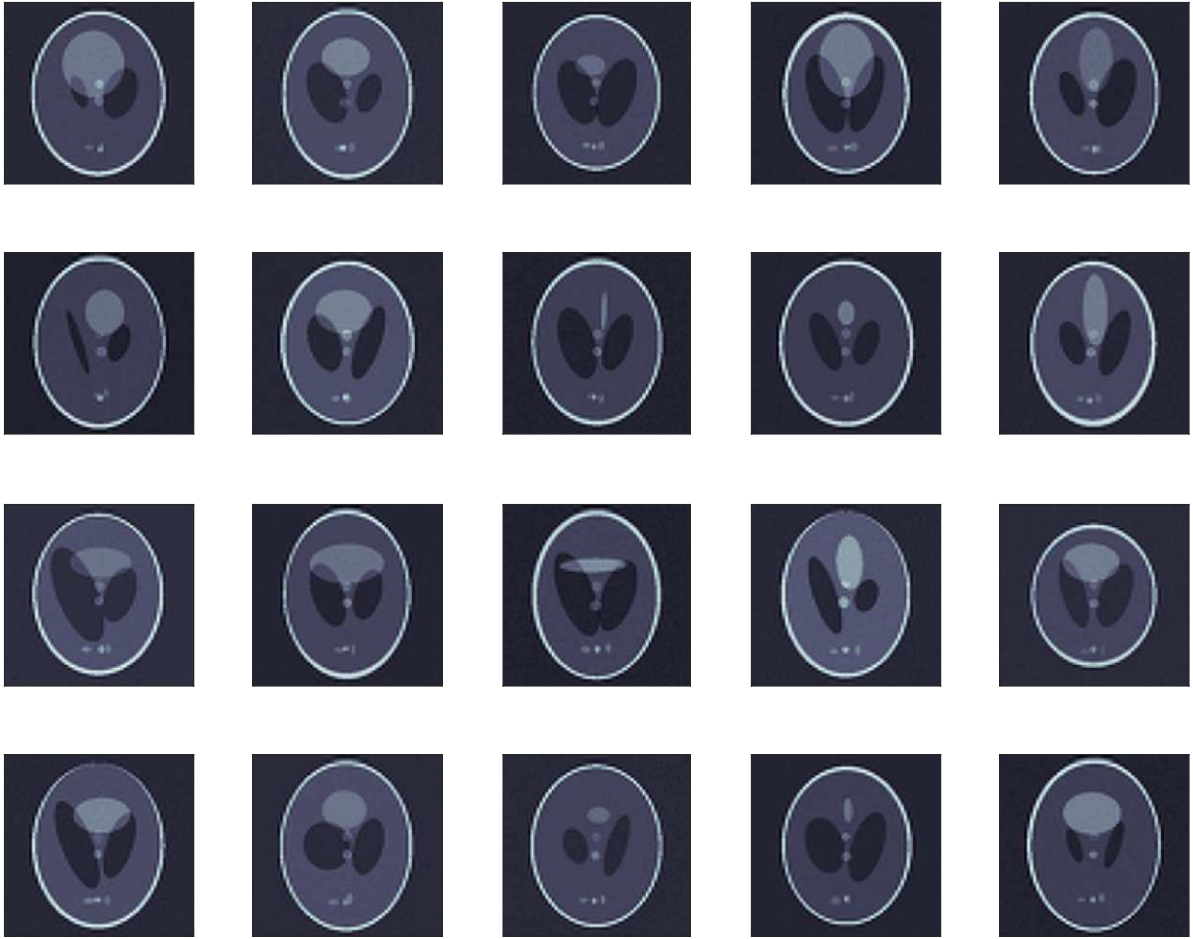}
\caption{Some observations from the training set in the SL phantom example.}
\label{fig:train_phantom}
\end{figure}

In  Figure~\ref{fig:sing_val_phantom}, we first inspect a suitable choice of $m$ in this example.
Specifically, in the left panel of Figure~\ref{fig:sing_val_phantom}, we compute the explained variance as $\sum_{j=1}^m\hat{\lambda}^p_j/\sum_{j=1}^{n}\hat{\lambda}^p_j$ for different choices of $m$ where $\hat{\lambda}^p_j$ are the eigenvalues of the covariance matrix corresponding to the observations $y_1,\ldots,y_n$, which are computed either by the classical PCA or its sparse variant. From such plot we see that our observations have a relatively low-rank structure, and in particular, already for $m=230$ the explained variance is over $0.99$.
In the right panel of Figure~\ref{fig:sing_val_phantom}, we compute the minimal singular value $\sigma_{\min}$ of the (regularized) system matrix $\hat{A}_{m,q}$ from \eqref{eq:GSFPCAsystem} in order to choose $m$ so that we have $\cos(\mathcal{F}_q,\hat{\mathcal{E}}^p_m)>0$, as suggested by Theorem \ref{thm:main}. Specifically, if least-squares is used to solve \eqref{eq:GSFPCAsystem}, we compute $\cos(\mathcal{F}_q,\hat{\mathcal{E}}^p_m)=\sigma_{\min}(\hat A_{m,q})=\lambda_m(\hat A_{m,q}^*\hat A_{m,q})^{1/2}$, while if ridge regression is used instead, as explained in Section~\ref{sub:rr}, we compute the minimal singular value of a regularized version of $\hat A_{m,q}$, i.e.~the square root of $\lambda_m(\hat A_{m,q}^*\hat A_{m,q}+ \lambda\mathrm{diag}(\hat{\lambda}^p_1,\ldots\hat{\lambda}^p_m)^{-1})$.
We see that  $\sigma_{\min}(\hat{A}_{m,q})$ approaches $\cos(\mathcal{F}_q,\mathcal{G}_p)=\sigma_{\min}(A_{p,q})=0$ as $m$ approaches $p$, where $A_{p,q}$ is the system matrix from \eqref{eq:GSsystem}, but crucially, for the choices of $m\leq500$ we have $\sigma_{\min}(\hat{A}_{m,q})>0.02$ in this example.  Interestingly, we see that for the intermediate choices of $m$, sparse PCA  provides certain regularization  since $\sigma_{\min}(\hat{A}_{m,q})$ is larger when sparse PCA is used in place of classical PCA to compute $\hat{\mathcal{E}}^p_m$.

\begin{figure}
\centering
\includegraphics[scale=0.65]{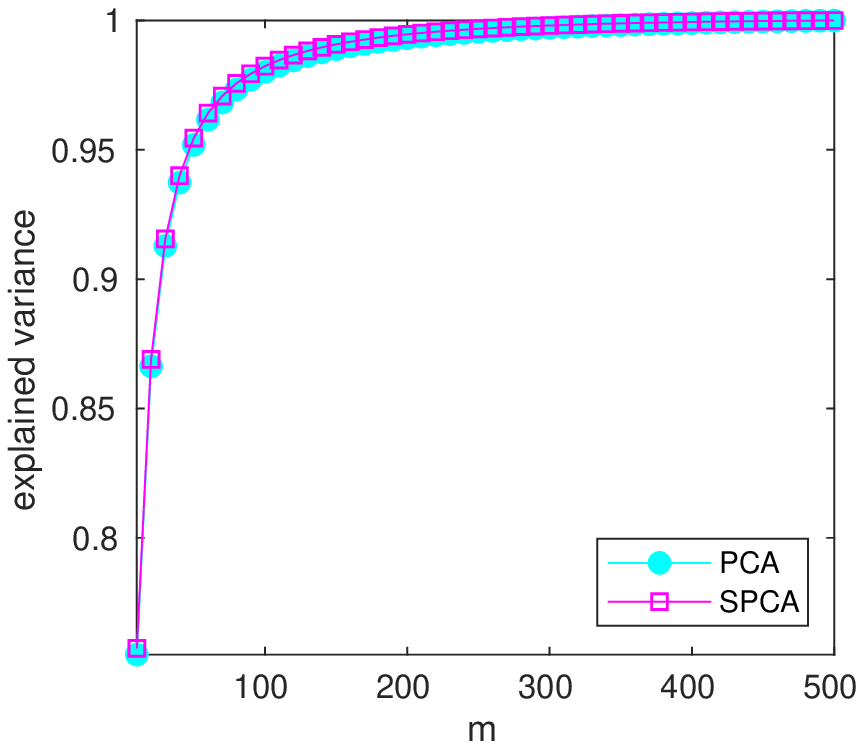} \includegraphics[scale=0.65]{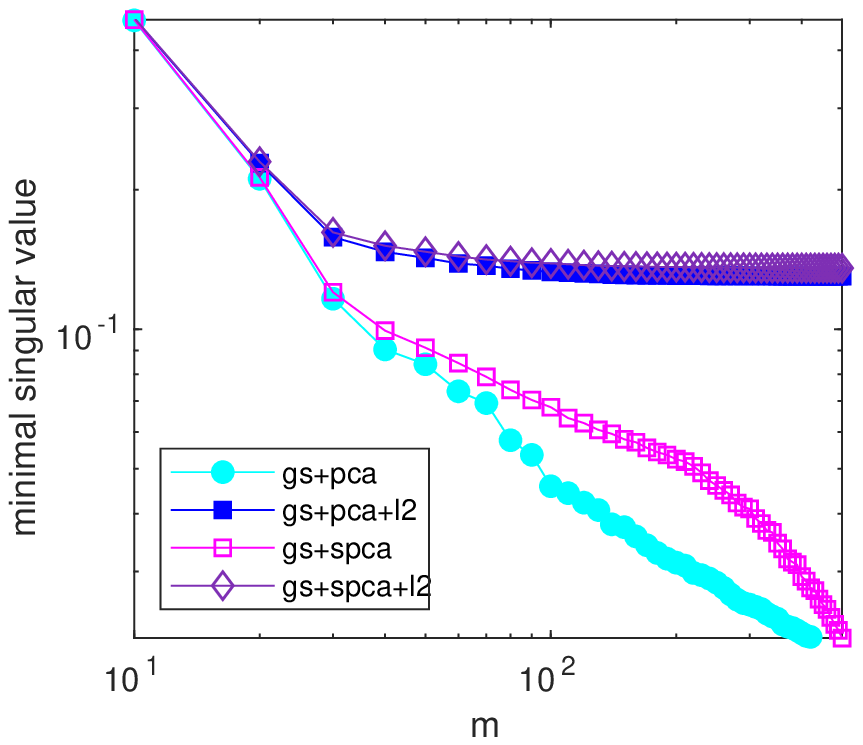}
\caption{Explained variance and minimal singular value of the (regularized) system matrix  from \eqref{eq:GSFPCAsystem} for different choices of $m\in\{10,20\ldots,500\}$ in the SL phantom example, where $p=4096$, $q=1024$, $n=512$ and (sparse) PCA is used to estimate $\hat{\mathcal{E}}^p_m$. Regularization parameter is $\lambda=0.0015$
.}
\label{fig:sing_val_phantom}
\end{figure}

Next, in Figure~\ref{fig:rec_phantom} we reconstruct the unseen phantom $f$ shown in the left panel of Figure~\ref{fig:orig_phantom} from its noisy $q=1024$ Fourier coefficients \eqref{eq:2dFourier_coef}. The desired resolution is the one corresponding to its $p$-dimensional wavelet projection shown in the right panel of Figure~\ref{fig:orig_phantom}, where $p=4096$ and wavelets are of order $s=4$. From the top panels of Figure~\ref{fig:rec_phantom}, we observe that without using the training observations, it is impossible to accurately reconstruct the phantom in the required wavelet resolution from given low-resolution Fourier measurements by GS (with either plain least-squares or its $\ell_1$ or $\ell_2$-regularizations). This is because $\sigma_{\min}(A_{p,q})=0$ for such choices of $p$ and $q$. However, if we compute $m=230$ eigenvectors from our $n=512$ training observations and reconstruct $f$ by $\hat{f}_{\mathrm{GS}\text{-}\mathrm{FPCA}}\in\hat{\mathcal{E}}^p_m\subseteq\mathcal{G}_p$, we can obtain much better reconstruction as shown in the bottom panels of Figure~\ref{fig:rec_phantom}. In particular, we see that an improved reconstruction  can be obtained when using sparse PCA instead of classical PCA to compute $\hat{\mathcal{E}}^p_m$ and when adding the $\ell_2$-regularization to the least-squares objective when computing the coefficients of  $\hat{f}_{\mathrm{GS}\text{-}\mathrm{FPCA}}$. 

\begin{figure}
\centering
\begin{tabular}{cc}
ground truth & projection onto $\mathcal{G}_p$ \\
\includegraphics[scale=0.85,trim={0.5cm 0.4cm 0.5cm 0.4cm},clip]{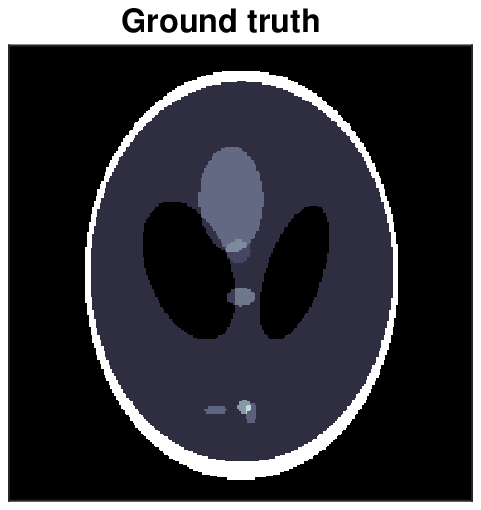} &
\includegraphics[scale=0.85,trim={0.5cm 0.4cm 0.5cm 0.4cm},clip]{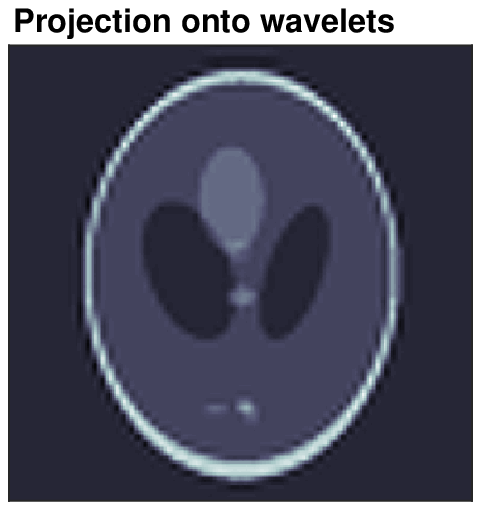}
\end{tabular}
\caption{The ground truth (left) of an unseen SL phantom and its projection (right) onto the space spanned by $p=4096$ DB4 wavelets, which is being reconstructed in Figures~\ref{fig:rec_phantom} and \ref{fig:rec_phantom_dif_sampling}.}
\label{fig:orig_phantom}
\end{figure}

\begin{figure}
\centering
\begin{tabular}{ccc}
GS & GS+$\ell_1$ & GS+$\ell_2$ \\
\includegraphics[scale=0.85,trim={0.5cm 0.4cm 0.5cm 0.4cm},clip]{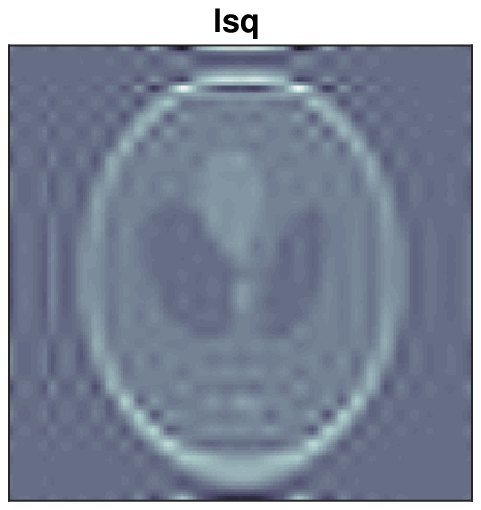} &
\includegraphics[scale=0.85,trim={0.5cm 0.4cm 0.5cm 0.4cm},clip]{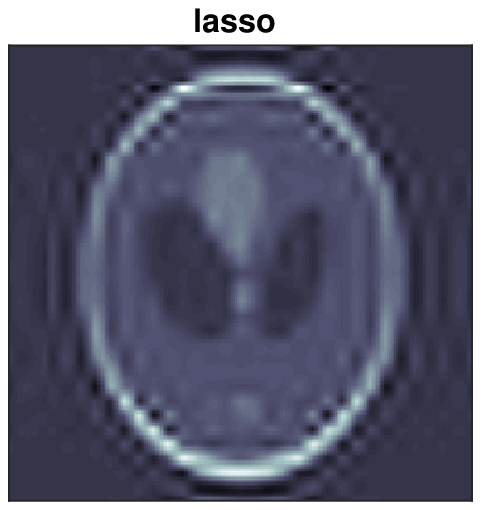} &
\includegraphics[scale=0.85,trim={0.5cm 0.4cm 0.5cm 0.4cm},clip]{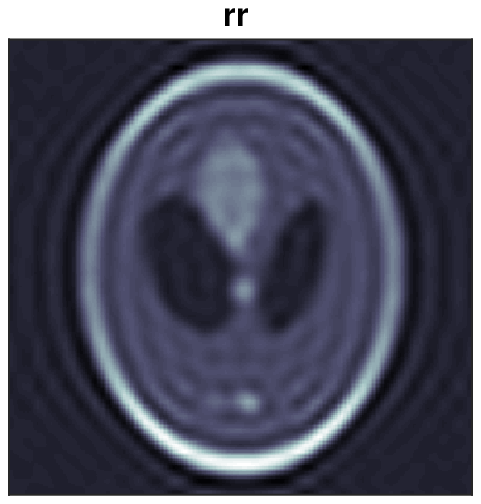} \\
GS+PCA & GS+SPCA & GS+SPCA+$\ell_2$ \\
\includegraphics[scale=0.85,trim={0.5cm 0.4cm 0.5cm 0.4cm},clip]{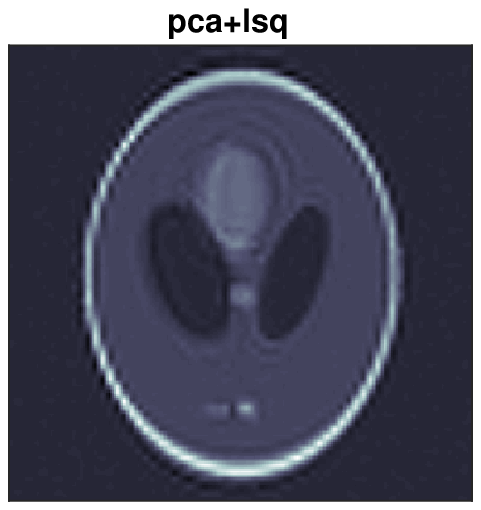} &
\includegraphics[scale=0.85,trim={0.5cm 0.4cm 0.5cm 0.4cm},clip]{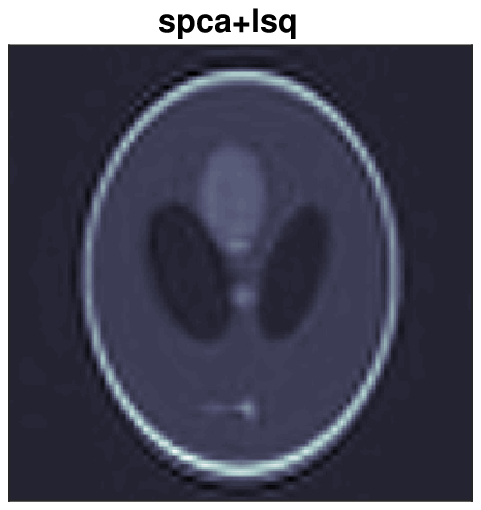} &
\includegraphics[scale=0.85,trim={0.5cm 0.4cm 0.5cm 0.4cm},clip]{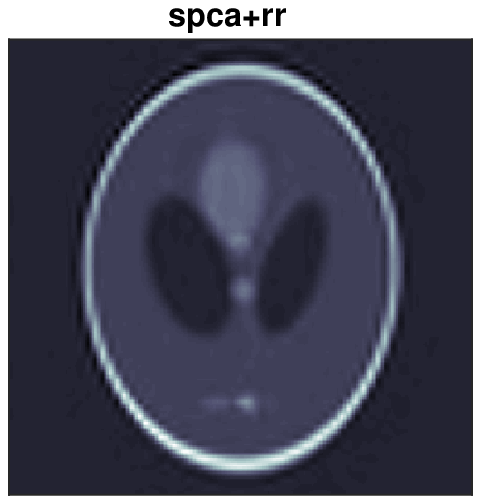} \\
\end{tabular}
\caption{Different reconstructions from $q=1024$ noisy Fourier measurements of the ground-truth SL phantom shown in Figure~\ref{fig:orig_phantom}. Top panels correspond to the (regularized) GS reconstructions computed directly in $\mathcal{G}_p$, which is spanned by $p=4096$ DB4 wavelets,  while bottom panels correspond to the (regularized) GS-FPCA reconstructions computed in $\hat{\mathcal{E}}^p_m\subseteq\mathcal{G}_p$, which is spanned by $m=230$ eigenvectors estimated using $n=512$ training observations, some of which are shown in Figure~\ref{fig:train_phantom}.}
\label{fig:rec_phantom}
\end{figure}

Finally, in Figure~\ref{fig:rec_phantom_dif_sampling} we demonstrate recovery of the same SL phantom shown in Figure~\ref{fig:orig_phantom}, but now from much smaller number of measurements $q=256$. Beside recovery from the noisy Fourier measurements (top panels), we also consider reconstructions from the noisy measurements taken with respect to a pixel basis (bottom panels), which corresponds to taking averages of $f$ over a rectangular grid. Specifically, the samples of $f$ with respect to the $q$-dimensional pixel basis are of the following form
$$
\langle f, \psi_{k,j} \rangle = \int_{0}^1 \int_{0}^1 f(u,v) \mathbb{1}_{\bigl[\frac{k-1}{\sqrt{q}},\frac{k}{\sqrt{q}}\bigr)} (u) \mathbb{1}_{\bigl[\frac{j-1}{\sqrt{q}},\frac{j}{\sqrt{q}}\bigr)} (v) \D u \D v,
$$
where $k,j=1,\ldots,\sqrt{q}$. 
From Figure~\ref{fig:rec_phantom_dif_sampling}, we see that by reconstructing in the $m$-dimensional space estimated via sparse PCA, $m=200$, GS-FPCA still produces relatively accurate reconstructions from such low resolution measurements, while GS does not stand a chance at such high resolution.

\begin{figure}
\centering
\begin{tabular}{ccc}
& GS+$\ell_2$ & GS+SPCA+$\ell_2$ \\
\rotatebox{90}{\hspace{0.1cm} from Fourier samples} & \includegraphics[scale=0.84,trim={0.5cm 0.4cm 0.5cm 0.4cm},clip]{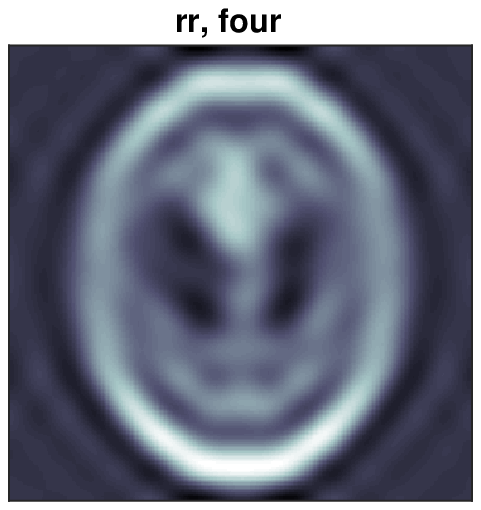} &
\includegraphics[scale=0.85,trim={0.5cm 0.4cm 0.5cm 0.4cm},clip]{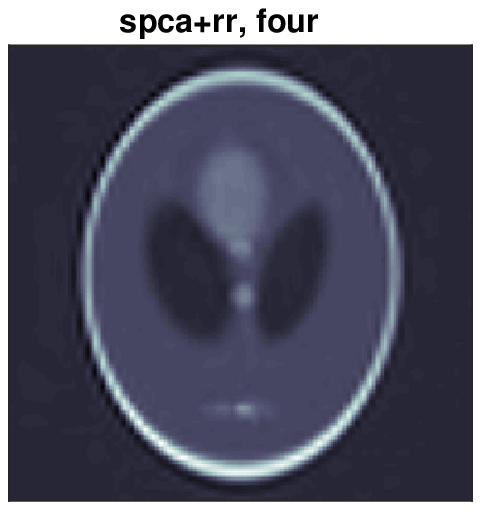} \\
\rotatebox{90}{\hspace{0.5cm} from pixel samples} & \includegraphics[scale=0.84,trim={0.5cm 0.4cm 0.5cm 0.4cm},clip]{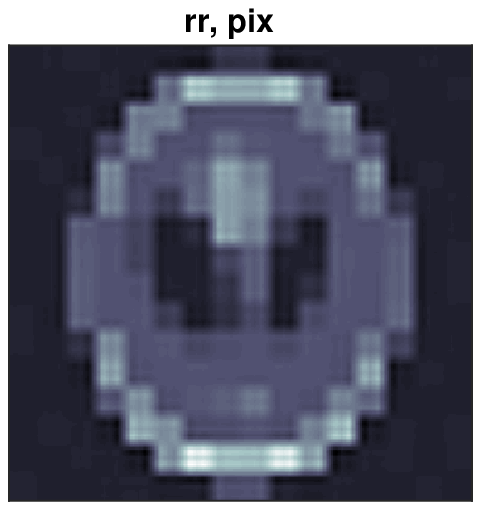} &
\includegraphics[scale=0.85,trim={0.5cm 0.4cm 0.5cm 0.4cm},clip]{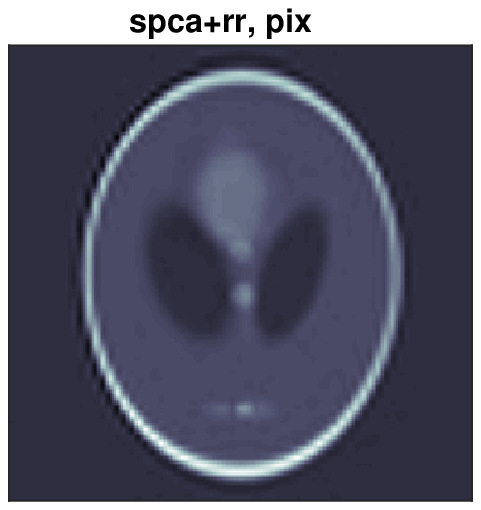} 
\end{tabular}
\caption{Different reconstructions from $q=256$ noisy Fourier (top) or pixel (bottom) measurements of the ground-truth SL phantom shown in Figure~\ref{fig:orig_phantom}. Left panels correspond to the regularized GS reconstructions in $\mathcal{G}_p$ spanned by $p=4096$ DB4 wavelets, while right panels correspond to the regularized GS-FPCA reconstructions  in $\hat{\mathcal{E}}^p_m\subseteq\mathcal{G}_p$ spanned by $m=200$ eigenvectors, which are estimated using $n=512$ training observations via sparse PCA.}
\label{fig:rec_phantom_dif_sampling}
\end{figure}

\section{Discussion and future work}\label{sec:conclusions}

In recent years, due to the development of deep neural networks (DNNs), there has been an increased interest in combining model-based and data-driven approaches for solving inverse problems. While 
promising results have been achieved 
empirically, theoretical understanding of such techniques is still largely lacking, e.g.~\cite{Schonlieb2019,Ravishankar2019}.  A particular instance of the inverse problem considered in this paper, corresponds to the inversion of a Fourier transform sampled up until a relatively low frequency $q$, which is an ill-posed problem typically studied from a model-based point of view, under the assumption that the unknown function is a sum of sparse 
spikes, e.g.~\cite{Blu2008,Candes2014}. 
In this paper, we approached such an inverse problem by invoking a training set and considering a data-driven technique based on FPCA, which is shown to be successful in high-resolution recovery provided appropriate low rank and angle conditions hold and provided the size of the training set $n$ is sufficiently large relative to the desired resolution $p$. Due to the flexibility to use sparse representations and thus sparse PCA, such procedure is particularly useful in a high-dimensional setting where $n$ is small relative to $p$. 

However,
provided $n$ is relatively large,  instead of FPCA,
it would be possible to use more expressive data-driven models based on DNNs to infer an optimal representation of the coefficients of $F$ with respect to $\mathcal{G}_p$. In particular, one could use autoencoders to learn a (nonlinear) decoding map $D:\mathbb{C}^m\mapsto\mathbb{C}^p$ and an encoding map $E:\mathbb{C}^p\mapsto\mathbb{C}^m$ such that $\sum_{i=1}^{n}\|D(E(y_i))-y_i\|_2^2$ is minimized. 
In the special case of a linear encoder and decoder with $D=E^{\top}$ and $D^{\top}D=I_m$, such procedure is equivalent to PCA, namely $D=(\hat{e}_1^p,\ldots,\hat{e}_m^p)$. Analogously to the framework considered in this paper, using (noisy Fourier) measurements $b\in\mathbb{C}^q$, one could then compute the desired coefficients as $\hat{\alpha}:=\argmin_{\alpha\in\mathbb{C}^m}\|A_{p,q}D(\alpha)-b\|_2^2$ and recover $\hat{F}:=(\varphi_1,\ldots,\varphi_p)D(\hat\alpha)$. We leave further consideration and analysis of such a non-linear method for future work.

In practice, there might be a need to reconstruct an object which only partially resembles training observations, while partially it contains structures unseen in training observations. For example, 
we might need to reconstruct a brain scan with a tumor dissimilar to anything contained in the training set of reconstructed brain scans. 
For this reason, it is important to investigate schemes for anomaly detection in the context where training observations are used for the reconstruction of an unseen object. 
Building on the framework developed in this paper, we could approach such problem 
by modeling the random field of interest as $G=F+H$, where $F\sim P$ and $H\sim Q$ and the first $m$ eigenfunctions associated to the measures $P$ and $Q$ are orthogonal. Once $G$ is estimated with respect to $\mathcal{E}_m$ from its measurements $b$, if the corresponding residual, $b-\hat{A}_{m,q}\hat{\alpha}$, is greater than the estimated level of noise, we could then either attempt to estimate $H$ in $\mathcal{G}_p$ from the residual,  or report an outlier and suggest increasing the number of measurements $q$.  We believe that by such a procedure, it would be also possible to further inform the correct specification of $m$, so that principal components greater than the noise level are not omitted from the reconstruction space.  We leave further investigation of such procedure for future work.

In this paper, we estimated FPCs from the high-resolution observations in $\mathcal{G}_p$, which could be recovered before hand from the high-resolution measurements with respect to $\mathcal{F}_r$ for a sufficiently large $r=r(p)$. However, in practice it may be more optimal to use such indirect measurements with respect to $\mathcal{F}_r$ to directly recover principal components in $\mathcal{G}_p$, which corresponds to an approach of estimating FPCs from indirect measurements recently studied in \cite{Lila2019}.

Finally, we mention that in this paper we assumed sampling with respect to a Riesz basis, which is an important generalization of an orthonormal basis in that it allows for more flexible sampling scenarios when measurements are acquired with respect to  a non-orthonormal basis. However, similarly as in \cite{AGH2014}, 
we believe that this  could be further relaxed by allowing the sampling system to constitute a frame, which would thus allow for nonuniform sampling patterns in the Fourier domain.


\appendix

\section{Proofs of theoretical results}
\label{app:proofs}

\begin{proof}[Proof of Lemma~\ref{lem:eig_mean_est}]
Observe that
$\sin\angle(\mathcal{E}_m,\hat{\mathcal{E}}^p_m)\leq \sin\angle(\mathcal{E}_m,\mathcal{E}^p_m) + \sin\angle(\mathcal{E}^p_m,\hat{\mathcal{E}}^p_m),$
where $\mathcal{E}^p_m:=\text{span}\{\phi_1^p,\ldots,\phi_m^p\}$ and $\{(\phi_j^p,\lambda_j^p)\}_{j=1}^p$ are the eigenfunction-eigenvalue pairs of the covariance operator $K_p$ associated to the random variable $Q_{\mathcal{G}_p}F$. Recall that $e^p_j:=(\langle \phi_j^p ,\varphi_1 \rangle,\ldots, \langle \phi_j^p ,\varphi_p \rangle)^{\top}$ is an eigenvector of $\Sigma_X$ with eigenvalue $\lambda_j^p$. 
%
%
To upper-bound $\sin\angle(\mathcal{E}_m,\mathcal{E}^p_m)$, first note that
for any $k,j=1,\ldots,p$, we have
\begin{align*}
\lambda_j^p \langle \phi_j^p ,\varphi_k \rangle 
&= \sum_{\ell=1}^p \Sigma_X^{(k\ell)}\langle \phi_j^p,\varphi_{\ell} \rangle 
= \int_{D}\int_{D} K(u,v) \overline{\varphi_k(u)} \phi_j^p(v) \D u \D v,
\end{align*}
where we used  $\lambda^p_j e^p_j=\Sigma_X e^p_j$ and $
\Sigma_X^{(k,\ell)}=\mathbb{E} [ \langle F-\mu, \varphi_k \rangle \overline{ \langle F-\mu, \varphi_{\ell}} \rangle]
= \sum_{j\in\mathbb{N}}\lambda_j \langle \phi_{j}, \varphi_k \rangle \overline{\langle \phi_{j}, \varphi_{\ell} \rangle} = \int_{D}\int_{D} K(u,v) \overline{\varphi_k(u)}\varphi_{\ell}(v) \D u \D v$, $k,l=1,\ldots,p$, as well as the fact that $\phi_j^p=\sum_{\ell=1}^p\langle \phi_j^p,\varphi_{\ell} \rangle\varphi_{\ell}$, 
respectively. Therefore
\begin{align*}
\lambda^p_j \int_{D} \phi^p_j(v) \overline{g(v)} \D v = \int_{D}\int_{D} K(u,v) \phi^p_j(u) \overline{g(v)} \D u \D v, \quad \forall g\in\mathcal{G}_p.
\end{align*}
Since also  $\lambda_j \int_{D} \phi_j(v) \overline{f(v)} \D v = \int_{D}\int_{D} K(u,v) \phi_j(u) \overline{f(v)} \D u \D v$, $\forall f\in\mathcal{L}_2(D;\mathbb{C})$, by  the approximation properties of the Galerkin method \cite{BO1987}, we have  $\lambda_j^p\leq \lambda_j$, and moreover, there exist $C$ (independent of $p$) and $p_0$ such that for any $p\geq p_0$ and $j=1,\ldots,m$, we have
\begin{align}
\label{BOresult2}
\lambda_j -\lambda_j^p &\leq C \lambda_j^2 \epsilon_{p}^2,\\
\label{BOresult1}
\|\phi_j - \phi_j^p\|  &\leq C  \epsilon_{p}.
\end{align}
Now, let  $e_j:=(\langle \phi_j ,\varphi_1 \rangle, \langle \phi_j ,\varphi_2 \rangle,\ldots)^{\top}$, $E_m := (e_1, \ldots, e_m) \in \mathbb{C}^{\infty\times m}$ and $E_m^p := (e_1^p, \ldots, e_m^p) \in \mathbb{C}^{p\times m}$. Let  $Q_p\in \mathbb{C}^{p\times \infty}$ denote the projection operator with identity $I_p$ constituting the first $p$ columns and the rest equal to zero, and let $\Theta(E_m^p,Q_pE_m)$ denote the $m\times m$ diagonal matrix whose $j$th diagonal entry is  $\arccos$ of the $j$th singular value of $(E_m^p)^{\top} Q_p E_m$. Observe that $\langle \phi_j,\phi_k^p \rangle =\sum_{\ell=1}^p \langle \phi_j,\varphi_{\ell} \rangle \langle \phi_k^p,\varphi_{\ell} \rangle = (e_k^p)^{\top} Q_p e_j$. 
Therefore,  we have  
\begin{align*}
\sin\angle(\mathcal{E}_m,{\mathcal{E}}^p_m) &= \|\sin \Theta(E_m^p,Q_p E_m)\|_{\mathrm{op}}
\leq \frac{1}{\sqrt{2}} \| E_m^p - Q_p E_m \|_{\mathrm{F}} \\
&\leq \sqrt{\frac{m}{2}} \max_{j=1,\ldots,m} \| e_j^p - Q_p e_j \|_2 
\leq \sqrt{\frac{m}{2}} \max_{j=1,\ldots,m} \| \phi^p_j - \phi_j \| 
\leq C \sqrt{\frac{m}{2}} \epsilon_{p},
\end{align*}
where in the last inequality we used \eqref{BOresult1}. To conclude  part $(a)$ of the proof, it remains to upper-bound $\sin\angle(\mathcal{E}^p_m,\hat{\mathcal{E}}^p_m)$. Similarly as above, let  $\hat E_m^p := (\hat e_1^p, \ldots, \hat e_m^p) \in \mathbb{C}^{p\times m}$, and let $\Theta(\hat E_m^p,E_m^p)$ denote the $m\times m$ diagonal matrix whose $j$th diagonal entry is  $\arccos$ of the $j$th singular value of $(\hat E_m^p)^{\top} E_m^p$.
Since $\langle\phi^p_j,\hat\phi^p_k\rangle = (\hat{e}^p_k)^{\top}e^p_j$, we have 
$\sin\angle({\mathcal{E}}^p_m,\hat{\mathcal{E}}^p_m)
= \|\sin \Theta(\hat E_m^p,E_m^p)\|_{\mathrm{op}}$, and since $\Sigma_X e^p_j=\lambda^p_j e^p_j$, we have $\Sigma_Ye_j^p=(\lambda_j^p+\tilde\sigma^2)e_j^p$. Thus, due to Davis--Kahan Theorem \cite{DK1970} and Weyl's inequality \cite{Weyl1912}, provided $\| \hat\Sigma_Y - \Sigma_Y \|_{\mathrm{op}}<(\lambda_m^p-\lambda_{m+1}^p)/2$ holds, we have 
\begin{equation*}
\sin\angle({\mathcal{E}}^p_m,\hat{\mathcal{E}}^p_m)
\leq  \frac{\sqrt{m}}{|\lambda_m^p+\tilde\sigma^2-\lambda_{m+1}(\hat\Sigma_Y)|} \| \hat\Sigma_Y - \Sigma_Y \|_{\mathrm{op}}
\leq  \frac{2\sqrt{m}}{\lambda_m^p-\lambda_{m+1}^p} \| \hat\Sigma_Y - \Sigma_Y \|_{\mathrm{op}}.
\end{equation*}
Due to result by \cite{koltchinskii2017b}, there exists $\tilde C$ so that for any $\delta\geq \exp(-n)$, the inequality
$$
\| \hat\Sigma_Y - \Sigma_Y \|_{\mathrm{op}} \leq \tilde{C} (\lambda_1^p+\tilde\sigma^2)(\sqrt{p/n}+\sqrt{\log(1/\delta)/n})
$$
holds with probability at least $1-\delta$, and thus, if $2\tilde{C} (\lambda_1^p+\tilde\sigma^2)(\sqrt{p/n}+\sqrt{\log(1/\delta)/n}) <\lambda_m^p-\lambda_{m+1}^p$, then 
\begin{align*}
\sin\angle({\mathcal{E}}^p_m,\hat{\mathcal{E}}^p_m) \leq 2 \tilde{C}\sqrt{m}(\lambda_m^p-\lambda_{m+1}^p)^{-1} (\lambda_1^p+\tilde\sigma^2)(\sqrt{p/n}+\sqrt{\log(1/\delta)/n}) 
\end{align*}
holds with probability at least $1-2\delta$. Moreover, due to \eqref{BOresult2} and the fact that $\lambda^p_{j}\leq\lambda_{j}$, we have $\lambda_m^p-\lambda_{m+1}^p \geq \lambda_m-\lambda_{m+1} - C \lambda_m^2 \epsilon_{p}^2$ and $ \lambda^p_1+\tilde\sigma^2 \leq  \lambda_1+\tilde\sigma^2$, so the result $(a)$ follows.
For part $(b)$, since $\mu_p=Q_{\mathcal{G}_p}\mu = (\varphi_1,\ldots,\varphi_p) \mu_X= (\varphi_1,\ldots,\varphi_p)\mu_Y$, we have 
\begin{align*}
 \|\mu -\hat\mu_p \| 
&\leq \|\mu - \mu_p \| + \|\mu_p - \hat\mu_p \| = \|Q_{\mathcal{G}_p^{\perp}}\mu \| + \|\hat{\mu}_{Y}-\mu_{Y}\|_2 \\
&\leq \|Q_{\mathcal{G}_p^{\perp}}\mu \| + \sqrt{n^{-1}\text{Tr}(\Sigma_X+\tilde\sigma^2 I_p)} + \sqrt{2n^{-1}(\lambda_1^p+\tilde\sigma^2)\log(1/\delta') },
\end{align*}
with probability at least $1-\delta'$, where in the last inequality we used the result by \cite{Joly2017}. The final result then follows by using that $\lambda^p_{1}\leq\lambda_{1}$. 
\end{proof}

\begin{proof}[Proof of Theorem~\ref{thm:main}]
First observe that for any $f\in\mathcal{L}_2(D;\mathbb{C})$ we have 
\begin{align*}
\|Q_{\mathcal{F}_q^{\perp}} f\|  
\leq  \|Q_{\mathcal{F}_q^{\perp}}  Q_{\mathcal{E}_m} f \| + \|Q_{\mathcal{F}_q^{\perp}} ( f- Q_{\mathcal{E}_m} f ) \|
\leq \|f\| \sup_{\substack{\{g\in\mathcal{E}_m: \|g\|=1\}}}  \| Q_{\mathcal{F}_q^{\perp}} g \| + \|Q_{\mathcal{E}_m^{\perp}} f \|.
\end{align*} 
Since $\cos\angle(\hat{\mathcal{E}}^p_m,\mathcal{F}_q)\geq1 - \sup_{\substack{\{f\in\hat{\mathcal{E}}^p_m: \|f\|=1\}}} \|Q_{\mathcal{F}_q^{\perp}} f\| $, by using the above inequality we get
\begin{align}\label{eq:sins}
\cos\angle(\hat{\mathcal{E}}^p_m,\mathcal{F}_q) 
&\geq 1 - \!\sup_{\substack{\{f\in\mathcal{E}_m: \|f\|=1\}}}\! \|Q_{\mathcal{F}_q^{\perp}} f\| - \!\sup_{\substack{\{f\in\hat{\mathcal{E}}^p_m: \|f\|=1\}}}\! \|Q_{\mathcal{E}_m^{\perp}} f\| \nonumber \\
&= 1 - \sin\angle(\mathcal{E}_m,\mathcal{F}_q) - \sin\angle(\mathcal{E}_m,\hat{\mathcal{E}}_m^p).
\end{align}
Define the event $\Omega:=\{\sin\angle(\mathcal{E}_m,\hat{\mathcal{E}}_m^p)\leq \tilde\epsilon_{mpn\delta}\}$, which due to Lemma~\ref{lem:eig_mean_est}$(a)$ is the event of probability at least $1-\delta$. 
Due to \R{eq:sins} and since $\sin\angle(\mathcal{E}_m,\mathcal{F}_q)<1-\tilde\epsilon_{mpn\delta}$, on $\Omega$ we have 
$
\cos\angle(\hat{\mathcal{E}}^p_m,\mathcal{F}_q) > 0.
$ 
 Now define $\tilde{F}_0= \sum_{j=1}^m \tilde\alpha_j \hat \phi_j^p$ such that
\begin{equation}\label{eq:tildealpha}
\{\tilde\alpha_j\}_{j=1}^{m} := \argmin_{\{\alpha_j\}_{j=1}^m\in\mathbb{R}^m} \sum_{k=1}^q \bigl| \langle F-\hat\mu_p,\psi_k \rangle -  \sum_{j=1}^m \alpha_j  \langle\hat\phi^p_j,\psi_k\rangle \bigr|^2.
\end{equation}
On $\Omega$, by the GS result \eqref{eq:GSbound} and bound \eqref{eq:sins},  we have
\begin{equation}\label{eq:gs_ineq}
 \|\tilde{F}_0 - (F-\hat\mu_p) \| 
 \leq \frac{\|Q_{\hat{\mathcal{E}}_m^p}(F -\hat\mu_p) - (F-\hat\mu_p)\|}{1 - \sin\angle(\mathcal{E}_m,\mathcal{F}_q) - \tilde\epsilon_{mpn\delta}} .
\end{equation}
Observe that
 \begin{align}\label{eq:approx_term}
 \|Q_{\hat{\mathcal{E}}_m^{p\perp}}(F -\hat\mu_p) \|
 &\leq  \|Q_{\hat{\mathcal{E}}_m^{p\perp}}(F -\mu) \| + \|Q_{\hat{\mathcal{E}}_m^{p\perp}}(\mu -\hat\mu_p) \|  \nonumber\\
 &\leq \|Q_{\hat{\mathcal{E}}_m^{p\perp}}Q_{\mathcal{E}_m}  (F -\mu)\| + \| Q_{\hat{\mathcal{E}}_m^{p\perp}} Q_{\mathcal{E}_m^{\perp}}  (F -\mu)  \| + \|\mu - \hat\mu_p \| \nonumber\\
 &\leq  \sin\angle(\mathcal{E}_m,\hat{\mathcal{E}}_m^p)\|F-\mu\| + \| Q_{\mathcal{E}_m^{\perp}}  (F-\mu)  \| +  \|\mu - \hat\mu_p \|.
\end{align}
Define the event $\Omega'=\{\|\mu -\hat\mu_p\|_2\leq  \bar\epsilon_{np\delta'} \}$, which due to Lemma \ref{lem:eig_mean_est}$(b)$ happens with probability $1-\delta'$.
Then, due to \R{eq:gs_ineq} and \R{eq:approx_term}, on $\Omega\cap\Omega'$ we have
\begin{equation}\label{eq:F0_bound}
 \|\tilde{F}_0 - (F-\hat\mu_p) \| 
 \leq \frac{  \tilde\epsilon_{mpn\delta}\|F-\mu\| + \| Q_{\mathcal{E}_m^{\perp}}  (F-\mu)  \| +\bar\epsilon_{np\delta'} }{1 - \sin\angle(\mathcal{E}_m,\mathcal{F}_q) - \tilde\epsilon_{mpn\delta}}.
\end{equation}
Finally, define 
$\Omega'':=\bigl\{ \|\hat \alpha - \tilde\alpha\|_2 \leq \sec\angle(\hat{\mathcal{E}}^p_m,\mathcal{F}_q) \sigma \sqrt{(2m+2\log(1/\delta'')) /(qr_1)} \bigr\}$, where vector $\tilde\alpha=(\tilde\alpha_1,\ldots,\tilde\alpha_m)^{\top}$ is defined as in \eqref{eq:tildealpha} and $\hat\alpha=(\hat\alpha_1,\ldots,\hat\alpha_m)^{\top}$ is as in \eqref{eq:hatalpha}.
On $\Omega\cap\Omega'$, the probability of $\Omega''$ conditional on $F_1,\ldots,F_n,Z_1,\ldots,Z_n$ (so that we are in the setting of a fixed design matrix) is at least $1 - \delta''$, due to the result from \cite{Hsu2012}. Also, since \R{eq:F0_bound} and 
\begin{align*}
\|\hat{F}-F\|  \leq \|\hat \alpha - \tilde\alpha\|_2 + \|\tilde{F}_0 - (F - \hat\mu_p) \|,
\end{align*}
the required bound holds on $\Omega\cap\Omega'\cap\Omega''$, which has the probability at least $1 - \delta - \delta' - \delta''$ because $\mathbb{P}(\Omega \cap \Omega' \cap \Omega'') \geq 1- \mathbb{P}((\Omega \cap \Omega')^{\mathrm{c}}) - \mathbb{E}\{\mathbb{P}(\Omega''^{\mathrm{c}}| F_1,\ldots,F_n,Z_1,\ldots,Z_n ) \mathbb{1}_{\Omega\cap\Omega'} \}$. 
\end{proof}

\section*{Acknowledgments} The author would like to thank Ben Adcock, Clarice Poon, Alberto Gil Ramos, Richard Samworth and Carola-Bibiane Sch\"onlieb for useful discussions and comments.

\section*{Declarations}

\subsection*{Funding}

The author was supported by an EPSRC grant  EP/N014588/1 for the Centre for Mathematical and Statistical Analysis of Multimodal Clinical Imaging.

\subsection*{Conflicts of interests}

The author declares that there is no conflict of interest.

\bibliographystyle{apalike}
\bibliography{references}

%
%
%
%
%
%
%
%
%

\end{document}